\documentclass[a4paper,11pt,reqno,noindent]{amsart}
\usepackage[centertags]{amsmath}
\usepackage{amsfonts,amssymb,amsthm,dsfont,cases,amscd,esint,enumerate,graphicx}
\usepackage[T1]{fontenc}
\usepackage[english]{babel}
\usepackage[applemac]{inputenc}
\usepackage{newlfont}
\usepackage{color}
\usepackage[body={15cm,21.5cm},centering]{geometry} 
\usepackage{fancyhdr}
\usepackage{enumerate}

\theoremstyle{plain}
\newtheorem{theor1}{Theorem}
\newtheorem{prop1}[theor1]{Proposition}
\newtheorem{cor1}[theor1]{Corollary}
\newtheorem{theor}{Theorem}[section]
\newtheorem{lem}[theor]{Lemma}
\newtheorem{prop}[theor]{Proposition}
\newtheorem{cor}[theor]{Corollary}
\theoremstyle{definition}
\newtheorem{rem}[theor]{Remark}
\theoremstyle{plain}

\newtheorem*{asn*}{\assumptionnumber}
  \providecommand{\assumptionnumber}{}
  \makeatletter
\newenvironment{asn}[2]
   {\renewcommand{\assumptionnumber}{Assumption \!(#1) {\normalfont--- #2}}
    \begin{asn*}
    \protected@edef\@currentlabel{{\normalfont(#1)}}}
   {\end{asn*}}
  \makeatother

\numberwithin{equation}{section}

\newcommand{\N}{\mathbb N}

\newcommand{\e}{\varepsilon}
\newcommand{\Pc}{\mathcal{P}}
\newcommand{\Bc}{\mathcal{B}}
\newcommand{\Ac}{\mathcal{A}}
\newcommand{\Ec}{\mathcal E}
\newcommand{\dist}{\operatorname{dist}}
\newcommand{\R}{\mathbb R}

\newcommand{\Ic}{\mathcal I}
\newcommand{\calF}{\mathcal F}
\newcommand{\Md}{\mathbb M}
\newcommand{\Nc}{\mathcal N}
\newcommand{\cvf}{\rightharpoonup}

\newcommand{\loc}{{\operatorname{loc}}}
\newcommand{\diam}{{\operatorname{diam}}}
\newcommand{\Id}{\operatorname{Id}}
\newcommand{\E}{\mathbb{E}}
\newcommand{\D}{\operatorname{D}}

\newcommand{\bb}{\bar{\boldsymbol b}}
\newcommand{\Bb}{\bar{\boldsymbol B}}
\newcommand{\Ld}{\operatorname{L}}

\newcommand{\Div}{{\operatorname{div}}}
\newcommand{\Sym}{{\operatorname{sym}}}
\newcommand{\Skew}{{\operatorname{skew}}}
\newcommand{\Tr}{{\operatorname{tr}}}
\newcommand{\step}[1]{\noindent \textit{Step} #1.}
\newcommand{\substep}[1]{\noindent \textit{Substep} #1.}
\newcommand{\Pm}{\mathbb{P}}
\newcommand{\expec}[1]{\mathbb{E}\left[ #1 \right]}
\newcommand{\expecm}[1]{\mathbb{E}\big[ #1 \big]}
\newcommand{\expecM}[1]{\mathbb{E}\bigg[ #1 \bigg]}

\usepackage[colorlinks,citecolor=black,urlcolor=black]{hyperref}

\title[Effective viscosity of random suspensions without uniform separation]{Effective viscosity of random suspensions\\without uniform separation}

\author[M. Duerinckx]{Mitia Duerinckx}
\address[Mitia Duerinckx]{Universit\'e Paris-Saclay, CNRS, Laboratoire de Math\'ematiques d'Orsay, 91400~Orsay, France \& Universit\'e Libre de Bruxelles, D\'epartement de Math\'ematique, 1050~Brussels, Belgium}
\email{mitia.duerinckx@u-psud.fr}

\begin{document}
\selectlanguage{english}

\begin{abstract}
This work is devoted to the definition and the analysis of the effective viscosity associated with a random suspension of small rigid particles in a steady Stokes fluid. While previous works on the topic have been conveniently assuming that particles are uniformly separated, we relax this restrictive assumption in form of mild moment bounds on interparticle distances.

\bigskip\noindent
{\sc MSC-class:} 35R60, 76M50, 35Q35, 76D07.
\end{abstract}

\maketitle

\setcounter{tocdepth}{1}
\tableofcontents

\vspace{-1cm}
\section{Introduction}
Consider a colloidal suspension of small rigid particles in a Stokes fluid.
Suspended particles act as obstacles, hindering the fluid flow and thus increasing the viscosity.
In a recent contribution~\cite{DG-19} with Gloria, we show in terms of homogenization theory that the suspension behaves at leading order like a Stokes fluid with some \emph{effective} viscosity, and in~\cite{DG-20b} we establish optimal error estimates. In~\cite{DG-20c}, we analyze the value of this effective viscosity in the low-density regime, in particular establishing the so-called Einstein formula and improving on several recent works on the topic~\cite{Niethammer-Schubert-19,GVH,GVM-20,GV-Hofer-20,GV-20}. In~\cite{DG-20a}, we further investigate the collective sedimentation of suspended particles under gravity.
In all those contributions, a crucial technical assumption is that particles are uniformly separated, which is necessary in various arguments, for instance when appealing to trace estimates and regularity theory at particle boundaries.
This separation assumption is however unsatisfactory from the physical viewpoint, as it is incompatible with the steady-state behavior, e.g.~\cite{BG-72a,BG-72}, and the present contribution aims at relaxing it as much as possible in form of mild inverse moment bounds on interparticle distances.
We focus on the definition of the effective viscosity and on the qualitative homogenization result, and we further provide general tools that can be used to adapt some more advanced results; see e.g.~\cite[Section~2]{DG-20c} and~\cite[Section~5]{GV-Hofer-20} on the validity of Einstein's formula in the low-density regime without uniform separation.

\medskip
In the case of smooth particles with some non-degeneracy condition, we essentially show in 3D that the effective viscosity is well-defined provided that $\expec{\rho^{-1}}<\infty$, where~$\rho$ stands for the distance between two neighboring particles, and we prove qualitative homogenization under the stronger condition $\expec{\rho^{-3/2}}<\infty$.
Although likely optimal in a general stationary ergodic setting, these moment bounds on interparticle distances are still quite restrictive and unphysical, \mbox{cf.~\cite{BG-72a,BG-72}}. We may draw the link with the well-known paradox of absence of solid-solid contacts in a 3D Stokes flow, which is related to flaws in the modeling:
real-life solid particles are slightly elastic, their boundary display some roughness, and no-slip boundary conditions are not exactly valid; see e.g.~\cite{GVH-12} and references therein.
Such corrections are not considered in the present contribution and we rather provide a detailed analysis of the ideal Stokes model.
In~\cite{DG21}, with Gloria, we investigate another line of research: under suitable mixing conditions, large clusters of close particles are unlikely in view of subcritical percolation, which can be exploited to prove homogenization without any condition on interparticle distances.
Finer geometric information might also be used in the spirit of~\cite{Girodroux-GV-21}.

\medskip
Our approach in this contribution is mainly inspired by the work of Jikov~\cite{Jikov-87,Jikov-90} on the homogenization problem for scalar elliptic equations with stiff inclusions; see also~\cite[Section~8.6]{JKO94}. In that scalar setting, however, required moment bounds on interparticle distances are much milder and only logarithmic moments are required in 3D. We emphasize two main differences:
\begin{enumerate}[---]
\item First, and most importantly, the incompressibility constraint in the present Stokes problem brings important rigidity and leads to completely different scalings. This is easily understood by noting that the incompressibility constraint can be eliminated by writing the Stokes equations as {\it fourth-order} elliptic equations on the vector potential; see e.g.~\cite{Francfort-92}.
As in~\cite{GVH-12}, spatial cut-offs in this situation are then naturally to be performed on the vector potential, so that one derivative of cut-off functions is lost with respect to scalar and compressible settings, which explains the different scalings; see the proof of Proposition~\ref{prop:ext-key}.
\smallskip\item Second, the vectorial character of the Stokes problem prohibits the use of scalar truncations: in contrast with e.g.~\cite[Section~8.6]{JKO94}, this forces us to appeal to the Sobolev embedding and further deteriorates the required moment conditions.
\end{enumerate}
We note some similarities with the homogenization problem for elliptic systems with degenerate random coefficients, e.g.~\cite{Chiarini-Deuschel-16,BFO-18,FHS-19,Bella-Schaffner-19}, where similar inverse moment conditions are required on coefficients.

\medskip
Before stating our main results, we close this introduction by recalling the formulation of the Stokes model for a viscous fluid in presence of a random suspension of small rigid particles, e.g.~\cite{DG-19}.
We denote by $d\ge 2$ the space dimension, and we consider a random ensemble of particles $\Ic=\bigcup_nI_n\subset\R^d$.
Stationarity, ergodicity, and regularity assumptions are postponed to Section~\ref{sec:main-res}.
In order to model a dense suspension of small particles, we rescale the random set~$\Ic$ by a small parameter $\e>0$ and consider $\e\Ic=\bigcup_n\e I_n$.
We then view these small particles $\{\e I_n\}_n$ as suspended in a solvent described by the steady Stokes equation: in a reference domain $U\subset\R^d$, given an internal force $f\in \Ld^2(U)^d$, the fluid velocity $u_\e\in H^1(U\setminus\e\Ic)^d$ satisfies
\begin{equation}\label{e.intro1}
-\triangle u_\e+\nabla S_\e=f,\qquad\Div (u_\e)=0,\qquad\text{in $U\setminus\e\Ic$},
\end{equation}
with $u_\e=0$ on $\partial U$.
(We implicitly assume here that no particle intersects the boundary.)
The pressure field is only defined up to an additive constant and we choose $S_\e\in\Ld^1(U\setminus\e\Ic)$ with $\int_{U\setminus\e\Ic}S_\e=0$.
Next, no-slip boundary conditions are imposed at particle boundaries: since particles are constrained to have rigid motions, this amounts to letting the velocity field~$u_\e$ be extended inside particles, $u_\e\in H^1(U)^d$, with the rigidity constraint
\begin{equation}\label{e.intro+2}
\D(u_\e)=0,\qquad\text{in $\e\Ic$},
\end{equation}
where $\D(u_\e)$ stands for the symmetric gradient of $u_\e$. In other words, this condition means that the velocity field $u_\e$ coincides with a rigid motion $x\mapsto V_{\e,n}+\Theta_{\e,n}x$ inside each particle $\e I_n$, for some $V_{\e,n}\in\R^d$ and some skew-symmetric matrix $\Theta_{\e,n}\in\R^{d\times d}$.
Finally, assuming that the particles have the same mass density as the fluid, or in the absence of gravity, buoyancy forces vanish, and
the force and torque balances on each particle take the form
\begingroup\allowdisplaybreaks
\begin{align}
\int_{\e\partial I_n}\sigma(u_\e,S_\e)\nu&=0,\\
\int_{\e\partial I_n}\Theta x\cdot\sigma(u_\e,S_\e)\nu&=0,\quad\text{for all skew-symmetric $\Theta\in\R^{d\times d}$},
\end{align}
\endgroup
where $\sigma(u_\e,S_\e)$ is the Cauchy stress tensor
\begin{equation}\label{eq:Cauchy}
\sigma(u_\e,S_\e)=2\D(u_\e)-S_\e\Id,
\end{equation}
and where $\nu$ stands for the outward unit normal vector at the particle boundaries.
These equations~\eqref{e.intro1}--\eqref{eq:Cauchy} have the following weak formulation,
\[2\int_U\D(g):\D(u_\e)=\int_Ug\cdot f,\qquad\text{$\forall\,g\in C^1_c(U)^d$: $\Div (g)=0$, $\D(g)|_{\e\Ic}=0$.}\]
This Stokes problem can also be viewed as a model for incompressible linear elasticity with stiff inclusions.

\subsection*{Notation}
\begin{enumerate}[\quad$\bullet$]
\item For vector fields $u,u'$ and matrix fields $T,T'$, we set $(\nabla u)_{ij}=\nabla_ju_i$, $\Div(T)=\nabla_jT_{ij}$, $T:T'=T_{ij}T'_{ij}$, $(u\otimes u')_{ij}=u_iu'_j$, where we systematically use Einstein's summation convention on repeated indices. For a matrix $E$, we write {$\nabla_Eu=E:\nabla u$}.
\smallskip\item For a velocity field~$u$ and pressure field $S$, we denote by $(\D(u))_{ij}=\frac12(\nabla_ju_i+\nabla_iu_j)$ the symmetric gradient and by $\sigma(u,S)=2\D(u)-S\Id$ the Cauchy stress tensor. At particle boundaries, we let $\nu$ denote the outward unit normal vector.
\smallskip\item We denote by $\Md^\Sym\subset\R^{d\times d}$ the subset of symmetric matrices, by $\Md_0^\Sym$ the subset of symmetric trace-free matrices, and by $\Md^\Skew$ the subset of skew-symmetric matrices.
We also write $\Ld^p(\R^d)^{d\times d}_\Sym=\Ld^p(\R^d;\Md^\Sym)$.
\smallskip\item We denote by $C\ge1$ any constant than only depends on the dimension $d$, on the reference domain~$U$, and on the parameters appearing in the different assumptions (in particular on~$\delta$ in~\ref{Hd}--\ref{Hd'} below). The value of the constant $C$ is allowed to change from one line to another.
We use the notation $\lesssim$ (resp.~$\gtrsim$) for $\le C\times$ (resp.~$\ge\frac1C\times$) up to such a multiplicative constant $C$. We add subscripts to $C,\lesssim,\gtrsim$ to indicate dependence on other parameters.
\smallskip\item The ball centered at $x$ of radius $r$ in $\R^d$ is denoted by $B_r(x)$, and we simply write $B(x)=B_1(x)$, $B_r=B_r(0)$, and $B=B_1(0)$.
\end{enumerate}

\medskip
\section{Main results}\label{sec:main-res}

We focus on the case $d>2$ for the statement of the main results, while the 2D case has some important difference and is briefly discussed in Remark~\ref{rem:2Dcase}.

\subsection{Assumptions}

We start with the construction and suitable assumptions on the random ensemble of particles.
Given an underlying probability space~$(\Omega,\Ac,\Pm)$,
let $\Pc=\{x_n\}_n$ be a random point process on $\R^d$, with a given enumeration, consider a collection of random shapes $\{I_n^\circ\}_n$, where each~$I_n^\circ$ is a connected random Borel subset of the unit ball~$B$,\footnote{Letting $\Bc(\R^d)$ denote the Borel $\sigma$-algebra on $\R^d$, we recall that a map $I^\circ:\Omega\to\Bc(\R^d):\omega\mapsto I^\circ(\omega)$ is a {\it random Borel subset} of $\R^d$ if the set $\{(\omega,x):\omega\in\Omega,x\in I^\circ(\omega)\}$ belongs to the product $\sigma$-algebra~$\Ac\times\Bc(\R^d)$, or alternatively if the indicator function $\mathds1_{I^\circ}$ is $(\Ac\times\Bc(\R^d))$-measurable on $\Omega\times\R^d$.} and define the corresponding random inclusions $I_n:=x_n+I_n^\circ$. We then consider the random set $\Ic:=\bigcup_nI_n$, which is assumed to satisfy the following general conditions, for some deterministic constant $\delta>0$.

\begin{asn}{H$_{\delta}^\circ$}{General conditions}\label{Hd}$ $
\begin{enumerate}[\quad$\bullet$]
\item \emph{Stationarity and ergodicity:}
The point process $\Pc=\{x_n\}_n$ and the associated random set $\Ic$ are stationary and ergodic.\footnote{Stationarity means that the laws of the translated point process $x+\Pc$ and of the translated random Borel set $x+\Ic$ do not depend of the shift $x\in\R^d$. Ergodicity then means that, if a measurable function of~$\Pc$ or~$\Ic$ is almost surely unchanged when $\Pc$ or~$\Ic$ is replaced by $x+\Pc$ or $x+\Ic$ for any $x\in\R^d$, then the function is almost surely constant.}
\smallskip\item \emph{Uniform $C^2$ regularity:}
Random shapes $\{I_n^\circ\}_n$ almost surely satisfy interior and exterior ball conditions with radius~$\delta$.
\smallskip\item \emph{Hardcore condition:}
There holds $\overline I_n\cap\overline I_m=\varnothing$ almost surely for all $n\ne m$.
\qedhere
\end{enumerate}
\end{asn}

When particles are close, not only their distance matters, but also the order of their quasi-contact.
We therefore need to refine the above hardcore condition,
and we focus on the case of smooth particles with uniformly non-osculating boundaries. This is expressed below in form of some ``parabolic'' version of a cone condition. While always satisfied in case of spherical particles, this excludes for instance the case of particles that would almost touch on flat components, as it would correspond to a contact of infinite order; see Figures~\ref{fig1}--\ref{fig2} below.
Note that our analysis is easily adapted to intermediate situations with contacts of any fixed order: this would lead to stronger moment conditions on interparticle distances and is not pursued here.

\medskip
Before we actually state relevant geometric conditions, we need to introduce some further notation.
First, we construct neighborhoods $\{I_n^+\}_n$ of the particles $\{I_n\}_n$ in form of truncated Voronoi cells,
\begin{equation}\label{eq:def-In+}
I_n^+\,:=\,(I_n+ B_\delta)\cap\Big\{x\in\R^d:\dist(x,I_n)<\inf_{m:m\ne n}\dist(x,I_m)\Big\}.
\end{equation}
In view of the uniform $C^2$ regularity of the particles, cf.~\ref{Hd}, it is easily checked that these neighborhoods $\{I_n^+\}_n$ are uniformly Lipschitz (with Lipschitz constant bounded by~$C/\delta$).
Next, we define ``model'' parabolic domains that are enclosed by close paraboloids with different radii: given a distance $\rho\ge0$ and radii $a_2>a_1>0$, we set
\begin{eqnarray}
\Gamma_{a_1,a_2}^+(\rho)&:=&B_\delta\cap\big\{(x_1,x')\in\R\times\R^{d-1}:-\rho+\tfrac1{a_2}|x'|^2<x_1<\tfrac1{a_1}|x'|^2\big\},\nonumber\\
\Gamma_{a_1,a_2}^-(\rho)&:=&B_\delta\cap\big\{(x_1,x')\in\R\times\R^{d-1}:-\rho-\tfrac1{a_1}|x'|^2<x_1<-\tfrac1{a_2}|x'|^2\big\},\label{eq:def-Gamma}
\end{eqnarray}
In these terms, we formulate the following geometric condition, for some deterministic constant $\delta>0$. It states that such parabolic domains can be included in the interparticle spacing $I_n^+\setminus I_n$ in the neighborhood of quasi-contact points, and the condition $\frac1{a_1}-\frac1{a_2}\ge\delta$ means that paraboloids can be chosen to be $\delta$-uniformly not osculating; see Figures~\ref{fig1}--\ref{fig2}.

\begin{asn}{H$_{\delta}'$}{Uniform non-degeneracy of contact points}\label{Hd'}$ $\\
For all $n$, for all $x\in\partial I_n$, there exists $0\le\rho\le\delta$, there exist radii $a_2>a_1\ge\delta$ with~\mbox{$\frac1{a_1}-\frac1{a_2}\ge\delta$}, and there exists a rotation $Q\in O(d)$, such that the rotated parabolic domain $x+Q\Gamma_{a_1,a_2}^+(\rho)$ or $x+Q\Gamma_{a_1,a_2}^-(\rho)$ is contained in $I_n^+\setminus I_n$.
\end{asn}

\begin{figure}
\begin{center}
\includegraphics[scale=1.5]{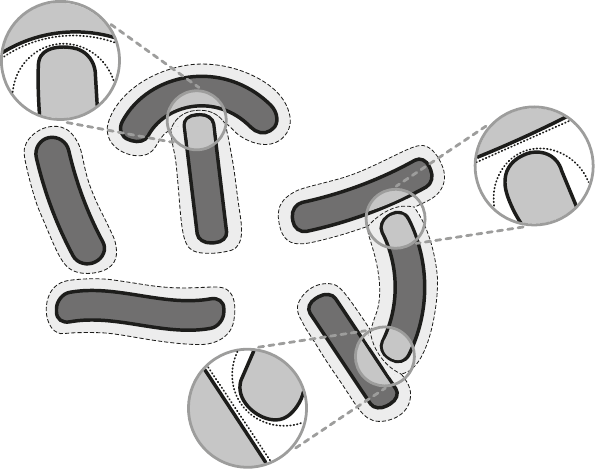}
\caption{\label{fig1} This displays a configuration with close particles satisfying Assumption~\ref{Hd'}. Disjoint neighborhoods~$\{I_n^+\}_n$ are represented as light gray areas around the particles. The zooms on the neighborhoods of quasi-contact points show that particle boundaries are not osculating, as prescribed by Assumption~\ref{Hd'}, with parabolic domains delimited by dotted lines.}
\end{center}
\begin{center}
\includegraphics[scale=1.5]{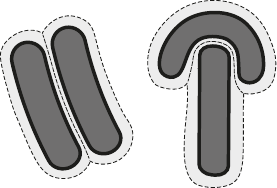}
\caption{\label{fig2} This displays examples of configurations of close particles that are forbidden by Assumption~\ref{Hd'} as their boundaries are osculating to infinite order.}
\end{center}
\end{figure}

Finally, we turn to assumptions on interparticle distances.
For all $n$, the (half) interparticle distance from $I_n$ is given by
\begin{equation}\label{eq:def-rn0}
\rho_n^\circ\,:=\,\min_{m:m\ne n}\tfrac12\dist(I_n,I_m).
\end{equation}
While previous works on the Stokes model~\eqref{e.intro1}--\eqref{eq:Cauchy} have focused on the convenient case of uniformly separated particles, that is, $\inf_n\rho_n^\circ>0$,
the present contribution aims at showing that this can be substantially weakened in form of mild inverse moment bounds.
For that purpose, under Assumption~\ref{Hd'}, we first need to introduce a better suited notion of interparticle distance $\rho_n\le\rho_n^\circ$:
for all $x\in\partial I_n$, we let~$\rho_n(x)$ denote the supremum of the admissible choices of~$\rho$ in~\ref{Hd'}, and we then define
\begin{equation}\label{eq:def-rhon}
\rho_n\,:=\,\inf_{x\in\partial I_n}\rho_n(x).
\end{equation}

\subsection{Construction of correctors}
We start with the definition of correctors for the Stokes problem~\eqref{e.intro1}--\eqref{eq:Cauchy}, thus adapting~\cite[Proposition~2.1]{DG-19} to the present setting without uniform particle separation.
The proof relies on the construction of a suitable admissible test function for the variational problem~\eqref{eq:var-psi} below, and we believe that the moment condition~\eqref{eq:mom-0} is optimal in general.
As is shown in the proof, existence and uniqueness of the corrector $\psi_E$ also hold under~\eqref{eq:mom-0} with $\eta=0$, but existence of a stationary pressure field is based on a weak compactness argument in~$\Ld^{1+}(\Omega)$ and therefore requires $\eta>0$.
Contacts between particles are allowed in dimension $d>5$ as no moment condition is required in that case.

\begin{theor1}[Correctors]\label{prop:cor}
Let $d>2$. On top of Assumptions~\ref{Hd} and~\ref{Hd'}, assume that interparticle distances $\{\rho_n\}_n$, cf.~\eqref{eq:def-rhon}, satisfy the following moment condition, for some~$\eta>0$,
\begin{equation}\label{eq:mom-0}
\begin{array}{lllll}
\text{for $d<5$}&:&\quad\textstyle\sum_n\expecm{\rho_n^{-\frac{5-d}2-\eta}\,\mathds1_{0\in I_n}}&<&\infty,\\
\vspace{-0.3cm}\\
\text{for $d=5$}&:&\quad\textstyle\sum_n\expecm{|\!\log\rho_n|^{1+\eta}\,\mathds1_{0\in I_n}}&<&\infty,
\end{array}
\end{equation}
while no moment condition is required in dimension $d>5$.
Then, for all $E\in\Md_0^\Sym$, there exists a unique minimizer $\D(\psi_E)$ of the variational problem
\begin{multline}\label{eq:var-psi}
\inf\Big\{\,\expec{|\!\D(\psi)+E|^2}~:~\psi\in\Ld^2(\Omega;H^1_\loc(\R^d)^{d}),~\text{$\nabla\psi$ stationary},\\
~\Div(\psi)=0,~(\D(\psi)+E)|_\Ic=0,~\expec{\D(\psi)}=0\,\Big\},
\end{multline}
and the minimum value defines a positive-definite symmetric linear map $\Bb$ on $\Md_0^\Sym$, which is the so-called effective viscosity,
\begin{equation}\label{eq:def-Bb}
E:\Bb E\,:=\,\expec{|\!\D(\psi_E)+E|^2}.
\end{equation}
Moreover, the minimizer $\D(\psi_E)$ can be characterized by the following PDE:
there exist a unique random vector field $\psi_E\in\Ld^2(\Omega;H^1_\loc(\R^d)^d)$, with anchoring $\int_B\psi_E=0$, and a unique associated pressure field $\Sigma_E\in\Ld^1(\Omega;\Ld^1_\loc(\R^d\setminus\Ic))$, such that
\begin{enumerate}[\quad$\bullet$]
\item the following equations are almost surely satisfied in the strong sense,
\begin{equation}\label{eq:cor}
\left\{\begin{array}{ll}
-\triangle\psi_E+\nabla\Sigma_E=0,&\text{in $\R^d\setminus\Ic$},\\
\Div(\psi_E)=0,&\text{in $\R^d$},\\
\D(\psi_E+Ex)=0,&\text{in $\Ic$},\\
\fint_{\partial I_n}\sigma(\psi_E+Ex,\Sigma_E)\nu=0,&\forall n,\\
\fint_{\partial I_n}\Theta(x-x_n)\cdot\sigma(\psi_E+Ex,\Sigma_E)\nu=0,&\forall n,\,\forall\Theta\in\Md^\Skew,
\end{array}\right.
\end{equation}
\item $\nabla\psi_E$ and $\Sigma_E\mathds1_{\R^d\setminus\Ic}$ are stationary, with the following estimates, for some $\eta>0$,
\[\begin{array}{rlllrrll}
\expecm{|\nabla\psi_E|^2}&\lesssim&|E|^2,
&\quad&\expecm{\nabla\psi_E}&=&0,\\
&&&&&&\vspace{-0.4cm}
\\
\expecm{|\Sigma_E|^{1+\eta}\mathds1_{\R^d\setminus\Ic}}&\lesssim&|E|^{1+\eta},
&\quad&\expecm{\Sigma_E\mathds1_{\R^d\setminus\Ic}}&=&0.
\end{array}\]
\end{enumerate}
In particular, the following convergences hold almost surely as $\e\downarrow0$,
\begin{equation}\label{eq:conv-cor}
\begin{array}{rllll}
(\nabla\psi_E)(\tfrac\cdot\e)&\cvf&0,&\quad\text{weakly} &\text{in $\Ld^2_\loc(\R^d)$},\\
(\Sigma_E\mathds1_{\R^d\setminus\Ic})(\tfrac\cdot\e)&\cvf&0,&\quad\text{weakly} &\text{in $\Ld^{1+\eta}_\loc(\R^d)$},\\
\e\psi_E(\tfrac\cdot\e)&\to&0,&\quad\text{strongly} &\text{in $\Ld^q_\loc(\R^d)$,~~for all $q<\frac{2d}{d-2}$.}
\end{array}
\qedhere
\end{equation}
\end{theor1}

In contrast with the case of uniformly separated particles, cf.~\cite[Proposition~2.1]{DG-19}, we emphasize that under the moment condition~\eqref{eq:mom-0} the pressure field $\Sigma_E$ above is only defined in $\Ld^{1+\eta}(\Omega)$ for some $\eta>0$, and not in~$\Ld^2(\Omega)$.
Improving on this integrability naturally requires a stronger moment condition, as shown in the following.

\begin{prop1}[Integrability of the pressure]\label{prop:Sig-int}
Let $d>2$ and let $\gamma:=\frac{2d(d+1)}{d^2+5d-2}$ for abbreviation. On top of Assumptions~\ref{Hd} and~\ref{Hd'}, given~$1<\alpha<2$, let one of the following conditions hold for interparticle distances~$\{\rho_n\}_n$:
\begin{enumerate}[\quad$\bullet$]
\item in case $\alpha\le\frac{d}{d-1}$ with $d\le5$, assume that
\begin{equation*}
\begin{array}{lllll}
\text{for $d<5$}&:&\quad\textstyle\sum_n\expecm{\rho_n^{-\frac{\alpha}{2-\alpha}\frac{5-d}2}\,\mathds1_{0\in I_n}}&<&\infty,\\
\vspace{-0.3cm}\\
\text{for $d=5$}&:&\quad\textstyle\sum_n\expecm{|\!\log\rho_n|^{\frac{\alpha}{2-\alpha}}\,\mathds1_{0\in I_n}}&<&\infty,
\end{array}
\end{equation*}
\item in case $\alpha<\gamma$ with $d>5$, no moment condition is required;
\smallskip\item in case $\frac{d}{d-1}\vee\gamma<\alpha<2$, with $\alpha\ne\frac{d}{d-2}$, assume that
\begin{equation*}
\textstyle\sum_n\expecm{\rho_n^{-\frac\alpha{2-\alpha}(\frac1\gamma-\frac1\alpha)(d+1)}\,\mathds1_{0\in I_n}}
\,<\,\infty.
\end{equation*}
\end{enumerate}
Then for all $E\in\Md_0^\Sym$ the pressure field $\Sigma_E$ constructed in Theorem~\ref{prop:cor} satisfies
\[\expecm{|\Sigma_E|^\alpha\mathds1_{\R^d\setminus\Ic}}\,\lesssim\,|E|^\alpha,\]
and there holds almost surely $(\Sigma_E\mathds1_{\R^d\setminus\Ic})(\tfrac\cdot\e)\cvf0$ weakly in $\Ld^\alpha_\loc(\R^d)$ as $\e\downarrow0$.
\end{prop1}

\subsection{Homogenization result}
We turn to the homogenization result for the Stokes problem~\mbox{\eqref{e.intro1}--\eqref{eq:Cauchy}}.
For that purpose, we first define admissible random ensembles of particles in a given bounded Lipschitz domain $U\subset\R^d$: the proof indeed requires to control the distance of particles to the boundary $\partial U$ similarly as interparticle distances.
We let~$\Nc_\e(U)\subset\N$ denote a random subset of indices such that
\[\big\{n:I_n\subset \tfrac1\e U,\,\dist(I_n,\partial \tfrac1\e U)\ge\delta\big\}\,\subset\,\Nc_\e(U)\,\subset\,\big\{n:I_n\subset\tfrac1\e U\big\},\]
and we define the associated random ensemble of particles in $U$,
\begin{equation}\label{eq:def-Ieps}
\Ic_\e(U)\,:=\,\bigcup_{n\in\Nc_\e(U)}\e I_n.
\end{equation}
In this setting, we consider corresponding neighborhoods $\{I_{n;U,\e}^+\}_n$ of the particles $\{I_n\}_n$,
\begin{equation*}
I_{n;U,\e}^+:=I_n^+\cap\tfrac1\e U,
\end{equation*}
we assume that Assumption~\ref{Hd'} holds with neighborhoods $\{I_n^+\}_n$ replaced by $\{I_{n;U,\e}^+\}_n$,
and we define the corresponding distances $\{\rho_{n;U,\e}\}_n$ as in~\eqref{eq:def-rhon}.

\medskip
With this notation, we may now formulate the homogenization result for~\mbox{\eqref{e.intro1}--\eqref{eq:Cauchy}}.
The proof is based on a div-curl argument together with an extension result for fluxes as inspired by the work of Jikov~\cite{Jikov-87,Jikov-90}. Due to non-uniform particle separation, extended fluxes are only controlled in $\Ld^\alpha$ for some integrability $\alpha<2$ depending on the moment condition on interparticle distances; see~Theorem~\ref{th:extension}. In view of the Sobolev embedding, Jikov's div-curl argument can then be performed provided $\alpha\ge\frac{2d}{d+2}$. This restriction leads to the moment condition~\eqref{eq:mom-2} below, which is expected to be optimal in general and coincides with the one in Proposition~\ref{prop:Sig-int} with $\alpha=\frac{2d}{d+2}$. We emphasize that this condition becomes more stringent in large dimension as the Sobolev exponent $\frac{2d}{d+2}$ increases to $2$.
Not surprisingly, the condition is stronger than the one for the existence of the corrector in Theorem~\ref{prop:cor} since defining correctors only requires to construct an admissible test function for the variational problem~\eqref{eq:var-psi}.

\begin{theor1}[Homogenization result]\label{th:main-qsh}
Let $d>2$. On top of Assumption~\ref{Hd}, given a bounded Lipschitz domain $U\subset\R^d$, let Assumption~\ref{Hd'} hold for $\{I_{n;U,\e}^+\}_n$, and assume that interparticle distances $\{\rho_{n;U,\e}\}_n$ satisfy almost surely
\begin{equation}\label{eq:mom-2}
\begin{array}{lllll}
\text{for }d=3&:&~~
\limsup_{\e\downarrow0}\,\e^d\sum_{n\in\Nc_\e(U)}({\rho_{n;U,\e}})^{-\frac32}\,&\hspace{-0.2cm}<\,\infty,\quad&\\
\vspace{-0.2cm}\\
\text{for }d\ge4&:&~~
\limsup_{\e\downarrow0}\,\e^d\sum_{n\in\Nc_\e(U)}({\rho_{n;U,\e}})^{-(\frac d2-1)}
\,&\hspace{-0.2cm}<\,\infty,\quad&
\end{array}
\end{equation}
where in case $d=6$ the exponent $\frac d2-1=2$ must be replaced by some exponent $>2$.
Denote by $\lambda:=\expec{\mathds1_\Ic}$ the volume fraction of the suspension,
let $\psi,\Sigma,\Bb$ be defined as in Theorem~\ref{prop:cor},
and define the following effective constant $\bb\in\Md_0^\Sym$: for all~$E\in\Md_0^\Sym$,
\begin{equation}\label{eq:def-bb}
\bb:E\,:=\,\frac1d\,\expecM{\sum_n\frac{\mathds1_{I_n}}{|I_n|}\int_{\partial I_n}(x-x_n)\cdot\sigma(\psi_E+Ex,\Sigma_E)\nu}.
\end{equation}
Given an internal force $f\in\Ld^2(U)^d$, let the velocity field $u_\e\in \Ld^2(\Omega;H^1_0(U)^d)$ and the associated pressure field $S_\e\in\Ld^1(\Omega;\Ld^1(U\setminus\Ic_\e(U)))$, with anchoring \mbox{$\int_{U\setminus\Ic_\e(U)}S_\e=0$}, be almost surely the unique solutions of the Stokes problem~\eqref{e.intro1}--\eqref{eq:Cauchy}, that~is,
\begin{equation}\label{eq:st-het}
\left\{\begin{array}{ll}
-\triangle u_\e+\nabla S_\e=f,&\text{in $U\setminus\Ic_\e(U)$},\\
\Div (u_\e)=0,&\text{in $U$},\\
\D(u_\e)=0,&\text{in $\Ic_\e(U)$},\\
\int_{\e\partial I_n}\sigma(u_\e,S_\e)\nu=0,&\forall n,\\
\int_{\e\partial I_n}\Theta(x-\e x_n)\cdot\sigma(u_\e,S_\e)\nu=0,&\forall n,\,\forall\Theta\in\Md^\Skew.
\end{array}\right.
\end{equation}
Then we have almost surely, as $\e\downarrow0$,
\[\begin{array}{rlll}
u_\e-\bar u&\cvf&0,&\text{weakly in $H^1_0(U)$},\\
(S_\e-\bar S-\bb:\D(\bar u))\mathds1_{U\setminus\Ic_\e(U)}&\cvf&0,&\text{weakly in $\Ld^\frac{2d}{d+2}(U)$},
\end{array}\]
where the limiting velocity field $\bar u\in H^1_0(U)^d$ and the associated pressure field $\bar S\in\Ld^2(U)$, with anchoring $\int_U\bar S=0$, are the unique solutions of the following homogenized equation,
\begin{equation}\label{eq:st-hom}
\left\{\begin{array}{ll}
-\Div(2\Bb\D(\bar u))+\nabla\bar S=(1-\lambda)f,&\text{in $U$},\\
\Div (\bar u)=0,&\text{in $U$}.
\end{array}\right.
\end{equation}
In addition, provided that $f\in\Ld^p(U)^d$ for some $p>d$, the following corrector results hold almost surely, as $\e\downarrow0$,
\begin{eqnarray}
\Big\|u_\e-\bar u-\sum_{E\in\Ec}\e\psi_E(\tfrac\cdot\e)\nabla_E\bar u\Big\|_{H^1(U)}&\to&0,\nonumber\\
\inf_{c\in\R}\Big\|S_\e-\bar S-\bb:\D(\bar u)-\sum_{E\in\Ec}\Sigma_E(\tfrac\cdot\e)\nabla_E\bar u-c\Big\|_{\Ld^\frac{2d}{d+2}(U\setminus\e\Ic)}&\to&0,\label{eq:conv2sc}
\end{eqnarray}
where $\Ec$ stands for an orthonormal basis of $\Md_0^\Sym$.
\end{theor1}

\subsection{Further technical tools}
On top of the definition of the effective viscosity and the above qualitative homogenization result,
we wish to further extend more advanced results such as the validity of Einstein's formula for the effective viscosity at low density~\cite{DG-20c,GV-Hofer-20}, optimal quantitative error estimates for homogenization~\cite{DG-20b}, and the analysis of sedimentation~\cite{DG-20a}.
To this aim, we provide a couple of technical tools for the analysis of suspensions without uniform separation. These tools are used in~\cite[Section~2]{DG-20c} and~\cite[Section~5]{GV-Hofer-20} for the validity of Einstein's formula.

\medskip
We start with the following extension result for fluxes in presence of rigid particles, which constitutes the main technical tool in our proof of Theorem~\ref{th:main-qsh}.
Starting from a notion of flux~$q$ that accounts for the behavior outside rigid particles, we construct an extension~$\tilde q$ that is defined nontrivially inside the particles in such a way that the continuity equation holds globally, cf.~\eqref{eq:cond-ext-todo0}.
For that purpose, one views the suspension of rigid particles as the limit of a suspension of droplets with diverging shear viscosity, and extended fluxes are then naturally defined as limits of corresponding fluxes; see~Remark~\ref{rem:extension}.
This construction is inspired by a corresponding scalar result by Jikov~\cite{Jikov-87,Jikov-90} in the context of scalar elliptic equations with stiff inclusions (see also~\cite[Section~3.5]{JKO94}), but additional care is needed here to deal with the incompressibility constraint.

\begin{theor1}[Extension of fluxes]\label{th:extension}
Let $d>2$, let Assumptions~\ref{Hd} and~\ref{Hd'} hold, and let a realization of the random set $\Ic$ be fixed.
Given $\beta\in(1,\infty)$ and $f\in\Ld^\beta_\loc(\R^d)^d$, let~$q\in \Ld^\beta_\loc(\R^d)^{d\times d}_\Sym$ with $\Tr (q)=0$ satisfy
\begin{equation}\label{eq:cond-ext}
2\int_{\R^d}\D(g):q=\int_{\R^d}g\cdot f,\qquad\text{$\forall\,g\in C^1_c(\R^d)^d$: $\Div (g)=0$, $\D(g)|_{\Ic}=0$.}
\end{equation}
Then, for all $\alpha,r$ chosen as follows,
\begin{equation}\label{eq:ch-param}
\begin{array}{lllll}
&&r&\ge&\frac\beta{\beta-1},\\
1&<&\alpha&\le&\beta\wedge\frac{dr\beta}{r(d-\beta)+d\beta},
\end{array}
\qquad\quad\text{with}\quad\left\{\begin{array}{ll}
r<\frac{d\beta}{\beta-d},&\text{if $\beta>d$},\\
r\ne\frac{d\beta}{d\beta-d-\beta},&\text{if $\beta>\frac{d}{d-1}$},\\
\alpha<\frac{d}{d-1},&\text{if $r=\frac{\beta}{\beta-1}$},
\end{array}\right.
\end{equation}
there exists an extension $\tilde q\in\Ld^\alpha_\loc(\R^d)^{d\times d}_\Sym$ with $\Tr(\tilde q)=0$, as well as an associated pressure field $\tilde S\in\Ld^\alpha_\loc(\R^d)$, such that
\begin{equation}\label{eq:cond-ext-todo0}
\tilde q|_{\R^d\setminus\Ic}=q|_{\R^d\setminus\Ic},\qquad\text{and}\qquad-\Div(2\tilde q-\tilde S\Id)=f,\quad\text{in $\R^d$},
\end{equation}
and such that the following estimate holds, for all $R\ge1$,
\begin{equation}\label{eq:ext-est-res}
\|\tilde q\|_{\Ld^\alpha(B_R)}+\|\tilde S-\textstyle\fint_{B_R}\tilde S\|_{\Ld^{\alpha}(B_R)}
\,\lesssim_{\alpha,\beta,r}\,\Lambda(B_R;r,\tfrac{\beta\alpha}{\beta-\alpha})\Big(\|f\|_{\Ld^\frac{d\beta}{d+\beta}(\widehat B_R)}+\|q\|_{\Ld^\beta(\widehat B_R\setminus\Ic)}\Big),
\end{equation}
where we have set $\widehat B_R:=B_R\cup\bigcup_{n:I_n\cap B_R\ne\varnothing}I_n^+$ and
\begin{equation}\label{eq:lambda}
\Lambda(D;r,p)\,:=\,\Big(|D|+\sum_{n:I_n\cap D\ne\varnothing}\mu_{r}(\rho_n)^p\Big)^\frac1p,
\end{equation}
in terms of
\begin{equation}\label{eq:def-mu}
\mu_{r}(\rho)\,:=\,\left\{\begin{array}{lll}
\rho^{\frac{d+1}{2r}-\frac32}&:&r>\frac{d+1}3,\\
|\!\log \rho|^\frac1r&:&r=\frac{d+1}3,\\
1&:&r<\frac{d+1}3.
\end{array}\right.\qedhere
\end{equation}
\end{theor1}

As applications of this extension result, we establish a trace estimate at particle boundaries and a version of Caccioppoli's inequality.

\begin{cor1}[Trace estimate]\label{lem:trace}
Let $d>2$, let Assumptions~\ref{Hd} and~\ref{Hd'} hold, and let a realization of the random set $\Ic$ be fixed.
Let the velocity field $u\in H^1_\loc(\R^d)^d$ and the associated pressure field $S\in\Ld^1_\loc(\R^d\setminus\Ic)$ satisfy the homogeneous Stokes problem
\begin{equation}\label{eq:uS}
\left\{\begin{array}{ll}
-\triangle u+\nabla S=0,&\text{in $\R^d\setminus\Ic$},\\
\Div(u)=0,&\text{in $\R^d$},\\
\D(u)=0,&\text{in $\Ic$},\\
\int_{\partial I_n}\sigma(u,S)\nu=0,&\forall n,\\
\int_{\partial I_n}\Theta(x-x_n)\cdot\sigma(u,S)\nu=0,&\forall n,\,\forall\Theta\in\Md^\Skew.
\end{array}\right.
\end{equation}
Then for all $n$ and $g\in W^{1,\infty}(I_n^+)^d$ we have for all $\eta>0$,
\begin{align*}
&\Big|\int_{\partial I_n}g\cdot\sigma(u,S)\nu\Big|~\lesssim_\eta~\|g\|_{W^{1,\infty}(I_n^+\setminus I_n)}\,\Big(\int_{I_n^+\setminus I_n}|\!\D(u)|^2\Big)^\frac12\\
&\hspace{7cm}\times\bigg\{\begin{array}{lll}
\rho_n^{\frac1{4d}(d+1)(d+2)-\frac52-\eta}&:&d\le6,\\
1&:&d>6.
\end{array}\qedhere
\end{align*}
\end{cor1}

\begin{cor1}[Caccioppoli's inequality]\label{lem:cacc}
Let $d>2$, let Assumptions~\ref{Hd} and~\ref{Hd'} hold, and let a realization of the random set $\Ic$ be fixed.
Then, for all $\eta>0$, there exists $s<\frac{2d}{d-2}$ such that any solution $(u,S)$ of the homogeneous Stokes problem~\eqref{eq:uS} satisfies, for all $R\ge5$ and~$K\ge1$,
\begin{align*}
&\hspace{-0.3cm}\Big(\fint_{B_R}|\nabla u|^2\Big)^\frac12\\
&\quad\,\lesssim_{s,\eta}\,\bigg(KR^{-1}\Big(\fint_{B_{2R}}\Big|u-\fint_{B_{2R}}u\Big|^s\Big)^\frac1s+\big(K^{-1}+R^{-\frac d2(\frac1s-\frac{d-2}{2d})}\big)\Big(\fint_{B_{2R}}|\nabla u|^2\Big)^\frac12\bigg)\\
&\hspace{4cm}\times\left\{\begin{array}{lll}
1+R^{-d}\sum_{n:I_n\cap B_{2R}\ne\varnothing}\rho_n^{\frac1{4}(d+1)(d+2)-\frac52d-\eta}&:&d\le5,\\
1+R^{-d}\sum_{n:I_n\cap B_{2R}\ne\varnothing}\rho_n^{1-\frac d2-\eta}&:&d>5.
\end{array}\right.\qedhere
\end{align*}
\end{cor1}

\medskip
\section{Extension of fluxes}

This section is devoted to the proof of Theorem~\ref{th:extension}.
The argument relies on the following local extension result for incompressible fields, which is of independent interest.

\begin{prop}\label{prop:ext-key}
Let $d>2$, let Assumptions~\ref{Hd} and~\ref{Hd'} hold, and let a realization of the random set $\Ic$ be fixed.
Let $1<s\le r<\infty$, with $r\ne\frac{ds}{d-s}$ if~$s<d$, and with $r<\frac{ds}{d+s-ds}$ if $s<\frac{d}{d-1}$. Then, for all $n$, there exists an extension operator $P_n$ such that 
for all~$g\in C^1_b(I_n)^d$ with $\Div(g)=0$ the extension $P_ng\in W^{1,s}_0(I_n^+)^d$ satisfies
\begin{equation}\label{eq:ext-prop}
\D(P_ng)|_{I_n}=\D(g),\qquad\text{and}\qquad\Div(P_ng)=0,\quad\text{in $I_n^+$},
\end{equation}
and for all $p\ge s\vee\frac{drs}{d(r-s)+rs}$, with $p>d$ if $r=s$,
\begin{align}\label{eq:final-est-trunc}
\|\nabla P_ng\|_{\Ld^s(I_n^+)}\,\lesssim_{p,r,s}\,\mu_{r}(\rho_n)\,\|\!\D(g)\|_{\Ld^p(I_n)},
\end{align}
where we recall the notation~\eqref{eq:def-mu} for $\mu_{r}$.
\end{prop}

For future reference, we also highlight the following key tool for pressure estimates. It follows from the above local extension result combined with a standard use of the Bogovskii operator.
Note that the restriction on the geometry of the domain~$D$ and the associated constant $K(D)$ can be refined as e.g.~in~\cite[Lemma~III.3.2 and Theorem~III.3.1]{Galdi}.

\begin{lem}\label{lem:Bog2}
Let $d>2$, let Assumptions~\ref{Hd} and~\ref{Hd'} hold, and let a realization of the random set~$\Ic$ be fixed.
Let $D\subset\R^d$ be a bounded Lipschitz domain that is star-like with respect to every point in some ball of radius $R_0$, and set $K(D):=\frac1{R_0}\diam(D)$.
Let~\mbox{$1<s\le r<\infty$}, with $r\ne\frac{ds}{d-s}$ if~$s<d$, and with $r<\frac{ds}{d+s-ds}$ if~$s<\frac{d}{d-1}$. Then, for all $h\in C_b(D)$ with~$\int_{D\setminus\Ic}h=0$, there exists $z\in W^{1,s}_0(D)^d$ such that
\[\D(z)|_{\Ic}=0,\qquad\text{and}\qquad\Div(z)=h\mathds1_{D\setminus\Ic},\quad\text{in $D$},\]
and for all $p\ge s\vee\frac{drs}{d(r-s)+rs}$, with $p>d$ if $r=s$,
\begin{equation}\label{eq:Bogo-bnd0}
\|\nabla z\|_{\Ld^s(D)}
\,\lesssim_{p,r,s}\,K(D)^{d+1}\Lambda(D;r,\tfrac{ps}{p-s})\,\|h\|_{\Ld^p(D\setminus\Ic)},
\end{equation}
where we recall the definition~\eqref{eq:lambda}--\eqref{eq:def-mu} of $\Lambda$.
\end{lem}

\subsection{Cut-off functions}\label{sec:prel-In+}
We start with the construction of suitable cut-off functions for the inclusions~$\{I_n\}_n$ in their neighborhoods $\{I_n^+\}_n$.
The open subsets $\{J_n^j\}_j$ in the statement below are neighborhoods of quasi-contact points, that is, neighborhoods where~$\partial I_n$ and $\partial I_n^+$ are very close; see~Figure~\ref{fig3}.
The proof is inspired by the work of Jikov on homogenization problems with stiff inclusions, e.g.~\cite[Section~3.2]{JKO94}, and is also analogous to computations by Gérard-Varet and Hillairet in~\cite{GVH-12} for the drag force on a sphere close to a wall.
This result is easily adapted beyond Assumption~\ref{Hd'} to cover higher-order quasi-contacts between the particles, then leading to a worse dependence on the distance~$\rho_n$.

\begin{lem}[Cut-off functions]\label{lem:wn}
Let Assumptions~\ref{Hd} and~\ref{Hd'} hold, and let a realization of the random set $\Ic$ be fixed.
For all $n$, there exists a function $w_n\in W^{1,\infty}_0(I_n^+;[0,1])$ with $w_n|_{I_n}=1$ such that for all~$r\ge1$,
\begin{equation}\label{eq:nablawn}
\|\nabla w_n\|_{\Ld^r(I_n^+)}\,\lesssim_r\,\left\{\begin{array}{lll}
\rho_n^{\frac{d+1}{2r}-1}&:&r>\frac{d+1}2,\\
|\!\log\rho_n|^\frac1r&:&r=\frac{d+1}2,\\
1&:&r<\frac{d+1}2,
\end{array}\right.
\end{equation}
and
\begin{equation}\label{eq:nabla2wn}
\|\nabla^2 w_n\|_{\Ld^r(I_n^+)}\,\lesssim_r\,\left\{\begin{array}{lll}
\rho_n^{\frac{d+1}{2r}-2}&:&r>\frac{d+1}4,\\
|\!\log\rho_n|^\frac1r&:&r=\frac{d+1}4,\\
1&:&r<\frac{d+1}4.
\end{array}\right.
\end{equation}
In addition, there is a collection $\{J_n^j\}_{j=1}^{M_n}$ of open subsets of the form $J_n^j=B(x_n^j,\frac1C\delta)\cap I_n^+$,
with $M_n\lesssim1$ and $\dist(J_n^j,J_n^k)\ge\frac1C\delta$ for all $j\ne k$, such that
\[\|w_n\|_{W^{2,\infty}(I_n^+\setminus\bigcup_{j=1}^{M_n}J_n^j)}\,\lesssim\,1,\]
and for all $r\ge1$,
\begin{equation}\label{eq:nabla2wnlin}
\max_{1\le j\le M_n}\||\cdot-x_n^j|\nabla^2w_n\|_{\Ld^r(J_n^j)}\,\lesssim_r\,\left\{\begin{array}{lll}
\rho_n^{\frac{d+1}{2r}-\frac32}&:&r>\frac{d+1}3,\\
|\!\log\rho_n|^\frac1r&:&r=\frac{d+1}3,\\
1&:&r<\frac{d+1}3.
\end{array}\right.\qedhere
\end{equation}
\end{lem}

\begin{figure}
\begin{center}
\includegraphics[scale=1.5]{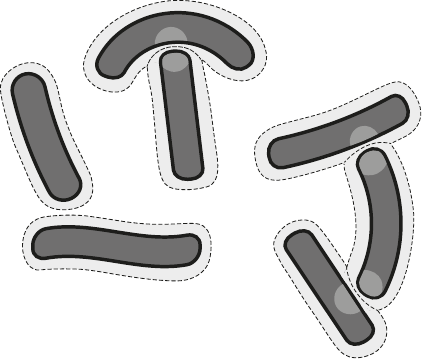}
\caption{\label{fig3}This displays a configuration of close particles. Disjoint neighborhoods $\{I_n^+\}_n$ are represented around the particles, and suitable neighborhoods $\{J_n^j\}_{j}$ of quasi-contact points are drawn in light gray.}
\end{center}
\end{figure}

\begin{proof}
Under Assumption~\ref{Hd'}, the construction of the neighborhoods $\{J_n^j\}_{j=1}^{M_n}$ is transparent, cf.~Figure~\ref{fig3}, and we only need to construct $w_n$ in one of those sets.
In view of the definition of the parabolic domains $\Gamma_{a_1,a_2}^\pm(\rho)$, cf.~\eqref{eq:def-Gamma}, it suffices to construct a cut-off function $w_{a_1,a_2}^\rho$ in $B_\delta\subset\R\times\R^{d-1}$ such that $w_{a_1,a_2}^\rho=0$ for $x_1<-\rho+\frac1{a_2}|x'|^2$ and $w_{a_1,a_2}^\rho=1$ for $x_1>\frac1{a_1}|x'|^2$.
By assumption we consider $a_2>a_1\ge\delta$ with $\frac1{a_1}-\frac1{a_2}\ge\delta$, and by scaling it suffices to consider $a_1=1$. More precisely, we consider the set
\[E=\big\{(x_1,x')\in\R\times\R^{d-1}:x_1\ge-1,~|x'|\le\tfrac12\big\},\]
and, given $\rho>0$ and $a>1$ with $1-\frac1a\ge\delta$, we construct a cut-off function $w_a^\rho\in C^{1,1}_b(E)$ such that $w_a^\rho=0$ for $x_1<-\rho+\frac1a|x'|^2$ and $w_a^\rho=1$ for $x_1>|x'|^2$, and such that for all~$r\ge1$,
\begingroup\allowdisplaybreaks
\begin{eqnarray}
\|\nabla w_a^\rho\|_{\Ld^r(E)}&\lesssim_r&\left\{\begin{array}{lll}
\rho^{\frac{d+1}{2r}-1}&:&r>\frac{d+1}2,\\
|\!\log\rho|^\frac1r&:&r=\frac{d+1}2,\\
1&:&r<\frac{d+1}2,
\end{array}\right.\label{eq:suff-B1B2glue1}\\
\|\nabla^2 w_a^\rho\|_{\Ld^r(E)}&\lesssim_r&\left\{\begin{array}{lll}
\rho^{\frac{d+1}{2r}-2}&:&r>\frac{d+1}4,\\
|\!\log\rho|^\frac1r&:&r=\frac{d+1}4,\\
1&:&r<\frac{d+1}4,
\end{array}\right.\label{eq:suff-B1B2glue2}\\
\||\cdot|\nabla^2 w_a^\rho\|_{\Ld^r(E)}&\lesssim_r&\left\{\begin{array}{lll}
\rho^{\frac{d+1}{2r}-\frac32}&:&r>\frac{d+1}3,\\
|\!\log\rho|^\frac1r&:&r=\frac{d+1}3,\\
1&:&r<\frac{d+1}3.
\end{array}\right.\label{eq:suff-B1B2glue3}
\end{eqnarray}\endgroup
In other words, we need to construct a suitable interpolation between $1$ and $0$ in the domain enclosed by the two parabolas,
\[\big\{(x_1,x')\in \R\times\R^{d-1}:-\rho+\tfrac1a|x'|^2< x_1<|x'|^2,~|x'|\le\tfrac12\big\}.\]
As we aim to construct a $C^{1,1}$ test function, we cannot use linear interpolation: instead of the linear function $h^0(t)=t$ with $h^0(0)=0$ and $h^0(1)=1$, we rather consider as in~\cite{GVH-12} the cubic function
\[h(t)\,:=\,t^2(3-2t),\]
with $h(0)=0$, $h(1)=1$, and $h'(0)=h'(1)=0$.
We then define
\[w_a^\rho(x)\,:=\,w_a^\rho(x_1,x')\,:=\,\left\{\begin{array}{lll}
0&:&x_1\le-\rho+\frac1a|x'|^2,\\
h\big(\tfrac1{\theta_a^\rho(x')}(\rho+x_1-\frac1a|x'|^2)\big)&:&-\rho+\frac1a|x'|^2\le x_1\le|x'|^2,\\
1&:&x_1\ge|x'|^2,
\end{array}\right.\]
where for abbreviation we denote by $\theta_a^\rho$ the horizontal distance between the two parabolas,
\[\theta_a^\rho(x')\,:=\,\rho+(1-\tfrac1a)|x'|^2.\]
We check that $w_a^\rho$ belongs to $C^{1,1}(E)$ and it remains to establish the estimates~\eqref{eq:suff-B1B2glue1}--\eqref{eq:suff-B1B2glue3}.
Recalling the assumption $1-\frac1a\ge\delta$, a direct computation shows that there holds for~$-\rho+\frac1a|x'|^2\le x_1\le|x'|^2$,
\begin{equation}\label{eq:key-est-w}
|\nabla w_a^\rho(x)|\lesssim (\rho+|x'|^2)^{-1},\qquad |\nabla^2 w_a^\rho(x)|\lesssim (\rho+|x'|^2)^{-2}.
\end{equation}
We start with the proof of~\eqref{eq:suff-B1B2glue1}.
Using~\eqref{eq:key-est-w}, evaluating the integral over $x_1$, and using radial coordinates, we find
\[\int_{E}|\nabla w_a^\rho|^r\,\lesssim_r\,\int_{|x'|\le\frac12}(\rho+|x'|^2)^{1-r}\,dx'\,\lesssim\,\int_0^\frac12\frac{s^{d-2}}{(\rho+s^2)^{r-1}}\,ds,\]
which proves~\eqref{eq:suff-B1B2glue1} after evaluating the integral. The proof of~\eqref{eq:suff-B1B2glue2} follows the same line and is skipped.

\medskip\noindent
We turn to the proof of~\eqref{eq:suff-B1B2glue3}. For $-\rho+\frac1a|x'|^2\le x_1\le|x'|^2$ and $|x'|\le\frac12$, we find
\begin{eqnarray*}
|x|\,\le\,|x_1|+|x'|\,\lesssim\,\rho+|x'|.
\end{eqnarray*}
Combining this with~\eqref{eq:key-est-w}, evaluating the integral over $x_1$, and using radial coordinates, we find
\begin{eqnarray*}
\int_E|\cdot|^r|\nabla^2w_a^\rho|^r&\lesssim_r&\rho^r\int_{|x'|\le\frac12}(\rho+|x'|^2)^{1-2r}\,dx'+\int_{|x'|\le\frac12}|x'|^r\,(\rho+|x'|^2)^{1-2r}\,dx'\\
&\lesssim&\rho^r\int_0^{\frac12}\frac{s^{d-2}}{(\rho+s^2)^{2r-1}}\,ds+\int_0^{\frac12}\frac{s^{d+r-2}}{(\rho+s^2)^{2r-1}}\,ds
\end{eqnarray*}
which proves~\eqref{eq:suff-B1B2glue3} after evaluating the integrals.
\end{proof}

\subsection{Proof of Proposition~\ref{prop:ext-key}}
Starting from a Stein extension of $g$, the argument relies on the cut-off function constructed in Lemma~\ref{lem:wn} in order to make this extension vanish on the boundary $\partial I_n^+$. A naïve cut-off would however break the incompressibility property and cause serious troubles especially close to quasi-contact points. Instead, in the spirit of~\cite{Francfort-92}, taking inspiration from calculations by Gérard-Varet and Hillairet in~\cite{GVH-12}, we take cut-offs at the level of the vector potential.
We split the proof into three main steps.

\medskip
\step1 Extension to $I_n^+$.\\
Given $g\in C^1_b(I_n)^d$ with $\Div (g)=0$, we construct an extension $P_n^1g\in H^1_0(I_n+ B)^d$ such that
\begin{equation}\label{eq:restrict-P1}
\D(P_n^1g)|_{I_n}=\D(g),\qquad\text{and}\qquad\Div(P_n^1g)=0,\quad\text{in $I_n+B$},
\end{equation}
and for all $1<s<\infty$,
\begin{equation}\label{eq:first-ext0}
\|\nabla P_n^1g\|_{\Ld^s(I_n+B)}\,\lesssim_{s}\,\|\!\D(g)\|_{\Ld^s(I_n)}.
\end{equation}
For that purpose, let us first choose $V_g\in\R^d$ and $\Theta_g\in\Md^\Skew$ such that Korn's inequality yields for all $1<s<\infty$,
\begin{equation}\label{eq:Korn-g}
\|g-V_g-\Theta_gx\|_{W^{1,s}(I_n)}\,\lesssim_s\,\|\!\D(g)\|_{\Ld^s(I_n)}.
\end{equation}
Next, in view of the $C^2$ regularity of $I_n$, cf.\@ Assumption~\ref{Hd}, we can choose a Stein extension $P_n^0g\in C^1_b(I_n+B)^d$ with $P_n^0g|_{I_n}=g-V_g-\Theta_gx$, such that for all~$s\ge1$,
\begin{equation*}
\|P_n^0g\|_{W^{1,s}(I_n+B)}\,\lesssim_{s}\,\|g-V_g-\Theta_gx\|_{W^{1,s}(I_n)},
\end{equation*}
and thus, by~\eqref{eq:Korn-g}, for all $1<s<\infty$,
\begin{equation}\label{eq:first-ext00}
\|P_n^0g\|_{W^{1,s}(I_n+B)}\,\lesssim_{s}\,\|\!\D(g)\|_{\Ld^{s}(I_n)}.
\end{equation}
It remains to apply a cut-off to $P_n^0g$ to make it vanish on the boundary $\partial(I_n+B)$ while keeping the properties in~\eqref{eq:restrict-P1}.
For that purpose, choose a cut-off function $\chi\in C^\infty_c(I_n+B)$ with $\chi|_{I_n}=1$ and $\|\chi\|_{W^{1,\infty}(I_n+B)}\lesssim1$. By a standard construction based on the Bogovskii operator, e.g.~\cite[Theorem~III.3.1]{Galdi}, since the following compatibility relation holds,
\begin{equation*}
\int_{(I_n+B)\setminus I_n}\Div(\chi P_n^0g)\,=\,-\int_{\partial I_n}g\cdot\nu\,=\,-\int_{I_n}\Div(g)\,=\,0,
\end{equation*}
there exists $z_n(g)\in H^1_0((I_n+B)\setminus I_n)^d$ such that
\[\Div(z_n(g))=\Div(\chi P_n^0g),\qquad\text{in $(I_n+B)\setminus I_n$},\]
and for all $1<s<\infty$,
\[\|\nabla z_{n}(g)\|_{\Ld^s((I_n+B)\setminus I_n)}\,\lesssim_s\,\|\Div(\chi P_n^0g)\|_{\Ld^s((I_n+B)\setminus I_n)}.\]
Expanding the divergence in the right-hand side of this estimate, and using~\eqref{eq:first-ext00}, we find for all~$1<s<\infty$,
\begin{equation}\label{eq:prebnd-P1}
\|\nabla z_{n}(g)\|_{\Ld^s((I_n+B)\setminus I_n)}\,\lesssim_s\,\|P_n^0g\|_{W^{1,s}(I_n+B)}\,\lesssim_s\,\|\!\D(g)\|_{\Ld^s(I_n)}.
\end{equation}
Now we define
\[P_n^1g\,:=\,\chi P_n^0g-z_n(g)~~\in~H^1_0(I_n+B)^d,\]
which indeed satisfies $P_n^1g|_{I_n}=P_n^0g|_{I_n}=g-V_g-\Theta_gx$ and $\Div(P_n^1g)=0$, hence~\eqref{eq:restrict-P1}.
In addition, combining~\eqref{eq:first-ext00} and~\eqref{eq:prebnd-P1} yields for all $1<s<\infty$,
\[\|\nabla P_n^1g\|_{\Ld^s(I_n+B)}\,\lesssim\,\|P_n^0g\|_{W^{1,s}(I_n+B)}+\|\nabla z_n(g)\|_{\Ld^s((I_n+B)\setminus I_n)}\,\lesssim_s\,\|\!\D(g)\|_{\Ld^s(I_n)},\]
that is,~\eqref{eq:first-ext0}.

\medskip
\step2 Matrix potential for $P_n^1g$.\\
We construct a matrix field $\sigma[P_n^1g]\in C^1(\R^d)^{d\times d}_\Skew$ that decays at infinity such that
\begin{equation}\label{eq:restrict-sigma}
\Div(\sigma[P_n^1g])|_{I_n^+}=P_n^1g|_{I_n^+},
\end{equation}
and such that for all $\frac{d}{d-1}<s<\infty$ and $p>d$,
\begin{eqnarray}\label{eq:bnd-sigm}
\|\nabla\sigma[P_n^1g]\|_{\Ld^s(\R^d)}&\lesssim_s&\|\!\D(g)\|_{\Ld^\frac{ds}{d+s}(I_n)},\nonumber\\
\|\nabla\sigma[P_n^1g]\|_{\Ld^\infty(\R^d)}&\lesssim_p&\|\!\D(g)\|_{\Ld^p(I_n)}.
\end{eqnarray}
For that purpose, we extend $P_n^1g$ by $0$ to $\R^d$, viewing it as a compactly supported element of $H^1(\R^d)^d$, and for all $i,j$ we define $\nabla\sigma_{ij}[P_n^1g]\in\Ld^2(\R^d)^d$ as the unique solution of
\begin{equation}\label{eq:def-sigma}
-\triangle \sigma_{ij}[P_n^1g]\,=\,\partial_i(P_n^1g)_j-\partial_j(P_n^1g)_i,\qquad\text{in $\R^d$}.
\end{equation}
In view of~\eqref{eq:first-ext0}, Calder\'on--Zygmund potential theory yields $\nabla\sigma_{ij}[P_n^1g]\in W^{1,s}(\R^d)^d$ for all $1<s<\infty$. Moreover, as $P_n^1g$ is compactly supported, Riesz potential theory ensures that $\sigma_{ij}[P_n^1g]$ can itself be uniquely chosen as a decaying element in $C^1(\R^d)$.
Uniqueness and the form of the right-hand side in~\eqref{eq:def-sigma} ensure that $\sigma[P_n^1g]$ is skew-symmetric.
Taking the divergence in~\eqref{eq:def-sigma}, and using that $\Div(P_n^1g)=0$, we find
\[-\triangle\Div(\sigma[P_n^1g])=-\triangle P_n^1g,\qquad\text{in $\R^d$},\]
which entails 
\begin{equation*}
\Div(\sigma[P_n^1g])=P_n^1g,
\end{equation*}
that is,~\eqref{eq:restrict-sigma}.
It remains to check~\eqref{eq:bnd-sigm}.
First, for all $\frac{d}{d-1}\le s<\infty$ and~$p>d$, the Sobolev embedding gives
\begin{eqnarray*}
\|\nabla\sigma[P_n^1g]\|_{\Ld^s(\R^d)}&\lesssim_s&\|\nabla^2\sigma[P_n^1g]\|_{\Ld^\frac{ds}{d+s}(\R^d)},\\
\|\nabla\sigma[P_n^1g]\|_{\Ld^\infty(\R^d)}&\lesssim_p&\|\nabla\sigma[P_n^1g]\|_{W^{1,p}(\R^d)}.
\end{eqnarray*}
Second, for all $1<s<\infty$, Calder\'on--Zygmund potential theory for~\eqref{eq:def-sigma} gives
\begin{eqnarray}\label{eq:bnd-re-P200}
\|\nabla^2\sigma[P_n^1g]\|_{\Ld^s(\R^d)}&\lesssim_s&\|\nabla P_n^1g\|_{\Ld^s(\R^d)},\nonumber\\
\|\nabla\sigma[P_n^1g]\|_{\Ld^s(\R^d)}&\lesssim_s&\|P_n^1g\|_{\Ld^s(\R^d)}.
\end{eqnarray}
Combining these two ingredients, appealing to Poincaré's inequality for $P_n^1g$ supported in~$I_n+B$, and using~\eqref{eq:first-ext0}, the claim~\eqref{eq:bnd-sigm} follows.
For future reference, we note that a similar argument also gives, for all $\frac{d}{d-1}<s<\infty$ and $p>d$,
\begin{eqnarray}\label{eq:bnd-re-P2}
\|P_n^1g\|_{\Ld^s(\R^d)}&\lesssim_s&\|\!\D(g)\|_{\Ld^\frac{ds}{d+s}(I_n)},\nonumber\\
\|P_n^1g\|_{\Ld^\infty(\R^d)}&\lesssim_p&\|\!\D(g)\|_{\Ld^p(I_n)}.
\end{eqnarray}

\medskip
\step3 Conclusion.\\
Recall the cut-off function $w_n\in H^1_0(I_n^+)$ that we have constructed in Lemma~\ref{lem:wn}, as well as the collection of neighborhoods of quasi-contact points $J_n^j=B(x_n^j,\frac1C\delta)\cap I_n^+$, \mbox{$1\le j\le M_n$}.
Recalling that $\dist(J_n^j,J_n^k)\ge\frac1C\delta$ for $j\ne k$, we further define enlarged neighborhoods
\[J_n^j\quad\subset\quad J_n^{j,+}:=B(x_n^j,\tfrac{6}{5C}\delta)\cap I_n^+\quad\subset\quad J_n^{j,++}:=B(x_n^j,\tfrac{7}{5C}\delta)\cap I_n^+,\]
which then satisfy $\dist(J_n^{j,++},J_n^{k,++})\ge\frac{1}{5C}\delta$ for $j\ne k$,
and we write for abbreviation
\[\textstyle J_n:=\bigcup_{j=1}^{M_n}J_n^j,\qquad J_n^+:=\bigcup_{j=1}^{M_n}J_n^{j,+},\qquad J_n^{++}:=\bigcup_{j=1}^{M_n}J_n^{j,++}.\]
We split the proof into two further substeps, first constructing the extension $P_ng$ close to quasi-contact points in $J_n^+$, and then completing the construction globally.

\medskip
\substep{3.1} Construction of $P_ng$ close to quasi-contact points.\\
Given $1<s\le r<\infty$ with $r\ne\frac{ds}{d-s}$ if $s<d$, and with $r<\frac{ds}{d+s-ds}$ if $s<\frac{d}{d-1}$,
we construct a vector field $P_n^2g\in H^1_0(I_n^+)^d$ such that
\begin{equation}\label{eq:Pn2-restr}
P_n^2g|_{I_n\cap J_n^{+}}=P_n^1g|_{I_n\cap J_n^+},\qquad\text{and}\qquad \Div(P_n^2g)=0,\quad\text{in $I_n^+$},
\end{equation}
and for all $p\ge s\vee \frac{drs}{d(r-s)+rs}$, with $p>d$ if $r=s$,
\begin{equation}\label{eq:bnd-key-P3}
\|\nabla P_n^2g\|_{\Ld^s(I_n^+)}\,\lesssim_{p,r,s}\,\mu_r(\rho_n)\,\|\!\D(g)\|_{\Ld^p(I_n)},
\end{equation}
where we recall the notation~\eqref{eq:def-mu} for $\mu_{r}$.
For all $j$ we first choose a smooth cut-off function $\chi_n^j\in C^\infty(I_n^+)$ such that
\[\chi_n^j|_{J_n^{j,+}}=1,\qquad\chi_n^j|_{I_n^+\setminus J_n^{j,++}}=0,\qquad\|\chi_n^j\|_{W^{2,\infty}(I_n^+)}\lesssim1.\]
Given a collection of matrices $\{\Theta_n^j\}_{j=1}^{M_n}\subset\Md^\Skew$ to be fixed later, we then define
\begin{equation}\label{eq:def-Pn3}
P_n^2g\,:=\,\sum_{j=1}^{M_n}\Div\big(w_n\chi_n^j(\sigma[P_n^1g]-\Theta_n^j)\big).
\end{equation}
By definition of $w_n$, this is supported in $I_n^+$ and satisfies, in view of~\eqref{eq:restrict-sigma},
\[P_n^2g|_{I_n\cap J_n^+}=\Div(\sigma[P_n^1g])|_{I_n\cap J_n^+}=P_n^1g|_{I_n\cap J_n^+}.\]
Moreover, since $\sigma[P_n^1g]-\Theta_n^j$ is skew-symmetric, we obviously have $\Div(P_n^2g)=0$.
It remains to estimate the norm of $\nabla P_n^2g$.
To this aim, using~\eqref{eq:restrict-sigma} again, we compute
\begin{equation*}
\nabla P_n^2g
\,=\,\sum_{j=1}^{M_n}\nabla(w_n\chi_n^jP_n^1g)+\sum_{j=1}^{M_n}\nabla\big((\sigma[P_n^1g]-\Theta_n^j)\nabla (w_n\chi_n^j)\big).
\end{equation*}
Expanding the gradients, smuggling in the weights $x\mapsto|x-x_n^j|$, and using Hölder's inequality, we find for all $r\ge s\ge1$,
\begin{multline}\label{eq:prebnd-Png00}
\|\nabla P_n^2g\|_{\Ld^s(I_n^+)}\,\lesssim\,\|\nabla P_n^1g\|_{\Ld^s(I_n^+)}
+\|\nabla(w_n\chi_n^j)\|_{\Ld^{r}(I_n^+)}\|(P_n^1g,\nabla\sigma[P_n^1g])\|_{\Ld^\frac{rs}{r-s}(I_n^+)}\\
+\sum_{j=1}^{M_n}\||\cdot-x_n^j|\nabla^2(w_n\chi_n^j)\|_{\Ld^r(J_n^{j,++})}\||\cdot-x_n^j|^{-1}(\sigma[P_n^1g]-\Theta_n^j)\|_{\Ld^\frac{rs}{r-s}(J_n^{j,++})},
\end{multline}
and thus, inserting the estimates of Lemma~\ref{lem:wn} for norms of the cut-off function $w_n$, and recalling the definition~\eqref{eq:def-mu} of $\mu_r$,
\begin{multline}\label{eq:prebnd-Png}
\|\nabla P_n^2g\|_{\Ld^s(I_n^+)}\,\lesssim_r\,\mu_r(\rho_n)\bigg(\|\nabla P_n^1g\|_{\Ld^s(I_n^+)}
+\|(P_n^1g,\nabla\sigma[P_n^1g])\|_{\Ld^\frac{rs}{r-s}(I_n^+)}\\
+\sum_{j=1}^{M_n}\||\cdot-x_n^j|^{-1}(\sigma[P_n^1g]-\Theta_n^j)\|_{\Ld^\frac{rs}{r-s}(J_n^{j,++})}\bigg).
\end{multline}
We estimate the right-hand side in two different ways, corresponding to two different choices of the constants $\{\Theta_n^j\}_{j=1}^{M_n}$ and allowing for complementary ranges of exponents.
\begin{enumerate}[$\bullet$]
\item {\it Case~1:} choosing $\Theta_n^j=0$ for all $j$, given $1< s< d$ and $r>\frac{ds}{d-s}$, with~$r<\frac{ds}{d+s-ds}$ if~$s<\frac{d}{d-1}$, we obtain for all $p\ge s\vee\frac{drs}{d(r-s)+rs}$,
\begin{equation}\label{eq:first-bnd0}
\qquad\|\nabla P_n^2g\|_{\Ld^s(I_n^+)}\,\lesssim_{r,s}\,\mu_r(\rho_n)\,
\|\!\D(g)\|_{\Ld^{p}(I_n)}.
\end{equation}
For that purpose, we appeal to Hardy's inequality in the following form, see e.g.~\cite[Sections~1.3 and~12.8]{Opic-Kufner-90}, for all $x_0\in\R^d$ and~$1\le p<d$,
\[\qquad\||\cdot-x_0|^{-1}\sigma[P_n^1g]\|_{\Ld^p(\R^d)}\,\lesssim_p\,\|\nabla\sigma[P_n^1g]\|_{\Ld^p(\R^d)}.\]
Choosing $\Theta_n^j=0$,
and inserting this estimate into~\eqref{eq:prebnd-Png}, we find for all $r\ge s\ge1$ with $\frac{rs}{r-s}<d$,
\begin{equation*}
\qquad\|\nabla P_n^2g\|_{\Ld^s(I_n^+)}\,\lesssim_{r,s}\,\mu_r(\rho_n)\,\Big(\|\nabla P_n^1g\|_{\Ld^s(\R^d)}+ \|(P_n^1g,\nabla\sigma[P_n^1g])\|_{\Ld^\frac{rs}{r-s}(\R^d)}\Big),
\end{equation*}
and the claim~\eqref{eq:first-bnd0} follows from~\eqref{eq:first-ext0}, \eqref{eq:bnd-sigm}, and~\eqref{eq:bnd-re-P2}.

\smallskip\item {\it Case~2:} choosing $\Theta_n^j=\sigma[P_n^1g](x_n^j)$ for all $j$, given~$1< s\le r<\infty$, with $r<\frac{ds}{d-s}$ if~$s<d$, we obtain for all $p\ge s\vee\frac{drs}{d(r-s)+rs}$, with $p>d$ if $r=s$,
\begin{equation}\label{eq:second-bnd0}
\qquad\|\nabla P_n^2g\|_{\Ld^s(I_n^+)}\,\lesssim_{p,r,s}\,
\mu_r(\rho_n)\,\|\!\D(g)\|_{\Ld^{p}(I_n)}.
\end{equation}
For that purpose, we appeal to Hardy's inequality in the following form, see e.g.~\cite[Sections~1.3 and~12.8]{Opic-Kufner-90}, for all $x_0\in I_n^+$ and~$d<p\le\infty$,
\[\qquad\||\cdot-x_0|^{-1}(\sigma[P_n^1g]-\sigma[P_n^1g](x_0))\|_{\Ld^p(I_n^+)}\,\lesssim_p\,\|\nabla\sigma[P_n^1g]\|_{\Ld^p(\R^d)}.\]
Choosing $\Theta_n^j=\sigma[P_n^1g](x_n^j)$, and inserting this estimate into~\eqref{eq:prebnd-Png}, we find for all $r\ge s\ge1$ with $\frac{rs}{r-s}>d$,
\begin{equation*}
\qquad\|\nabla P_n^2g\|_{\Ld^s(I_n^+)}\,\lesssim_{r,s}\,\mu_r(\rho_n)\Big(\|\nabla P_n^1g\|_{\Ld^s(\R^d)}
+\|(P_n^1g,\nabla\sigma[P_n^1g])\|_{\Ld^\frac{rs}{r-s}(\R^d)}\Big),
\end{equation*}
and the claim~\eqref{eq:second-bnd0} follows from~\eqref{eq:first-ext0}, \eqref{eq:bnd-sigm}, and~\eqref{eq:bnd-re-P2}.
\end{enumerate}

\noindent
Combining~\eqref{eq:first-bnd0} and~\eqref{eq:second-bnd0}, and choosing the constants $\{\Theta_n^j\}_{j=1}^{M_n}$ accordingly in the definition~\eqref{eq:def-Pn3} of $P_n^2g$, the claim~\eqref{eq:bnd-key-P3} follows.

\medskip
\substep{3.2} Construction of $P_ng$ away from contact points.\\
Let $1<s\le r<\infty$ be fixed, with $r\ne\frac{ds}{d-s}$ if $s<d$, and with $r<\frac{ds}{d+s-ds}$ if $s<\frac{d}{d-1}$.
Choosing a cut-off function $\zeta\in C^\infty_c(I_n^+\setminus J_n)$ with $\zeta|_{I_n\setminus J_n^+}=1$ and $\|\zeta\|_{W^{1,\infty}(I_n^+\setminus J_n)}\lesssim1$,
we consider the vector field
\[Q_ng\,:=\,\zeta(P_n^1g-P_n^2g),\]
and we note that in view of~\eqref{eq:Pn2-restr} it satisfies
\begin{equation}\label{eq:def-Qnrestr}
Q_ng|_{I_n}=(P_n^1g-P_n^2g)|_{I_n},
\end{equation}
hence in particular $\Div(Q_ng)|_{I_n}=0$.
As this yields the following relation,
\[\int_{I_n^+\setminus (I_n\cup J_n)}\Div(Q_ng)\,=\,-\int_{\partial (I_n\setminus J_n)}(Q_ng)\cdot\nu\,=\,-\int_{I_n\setminus J_n}\Div(Q_ng)\,=\,0,\]
we can appeal to the same construction based on the Bogovskii operator as in Step~1: there exists $t_n(g)\in H^1_0(I_n^+\setminus (I_n\cup J_n))^d$ such that
\[\Div(t_n(g))=\Div( Q_ng),\qquad\text{in $I_n^+\setminus (I_n\cup J_n)$,}\]
and
\begin{equation}\label{eq:bnd-wn}
\|\nabla t_n(g)\|_{\Ld^s(I_n^+\setminus (I_n\cup J_n))}\,\lesssim_s\,\|\Div( Q_ng)\|_{\Ld^s(I_n^+\setminus J_n)}.
\end{equation}
Here comes the restriction to $d>2$ as the set $I_n^+\setminus(I_n\cup J_n)$ is typically not connected in dimension $d=2$; see Remark~\ref{rem:2Dcase} below.
In these terms, we finally define
\[P_ng\,:=\,P_n^2g+Q_ng-t_n(g)~~\in~H^1_0(I_n^+),\]
which satisfies, in view of~\eqref{eq:def-Qnrestr},
\[P_ng|_{I_n}=P_n^2g|_{I_n}+(P_n^1g-P_n^2g)|_{I_n}=P_n^1g|_{I_n},\]
and also $\Div(P_ng)=0$ by definition of $t_n(g)$.
In addition, combining~\eqref{eq:bnd-wn} with the definition of $Q_ng$, we find
\begin{equation*}
\|\nabla P_ng\|_{\Ld^s(I_n^+)}\,\lesssim_s\,\|\nabla P_n^2g\|_{\Ld^s(I_n^+)}+\|\nabla Q_ng\|_{\Ld^s(I_n^+)}
\,\lesssim\,\|(P_n^1g,P_n^2g)\|_{W^{1,s}(I_n^+)},
\end{equation*}
hence, using Poincaré's inequality and inserting~\eqref{eq:first-ext0} and~\eqref{eq:bnd-key-P3}, for all \mbox{$p\ge s\vee \frac{drs}{d(r-s)+rs}$}, with $p>d$ if $r=s$,
\begin{equation*}
\|\nabla P_ng\|_{\Ld^s(I_n^+)}\,\lesssim_{p,r,s}\,\mu_r(\rho_n)\,\|\!\D(g)\|_{\Ld^p(I_n)}.
\end{equation*}
This concludes the proof.
\qed

\begin{rem}[2D case]\label{rem:2Dcase}
The restriction to $d>2$ is due to the impossibility to fix a stream function $\sigma[P_n^1g]$ that would vanish at all quasi-contact points. More precisely, in Case~2 of the above proof, we adapt the stream function $\sigma[P_n^1g]$ locally by making it vanish at each quasi-contact point (cf.~choice of $\Theta_n^j$ in~\eqref{eq:def-Pn3}), and modifications are then glued together in $I_n^+\setminus J_n$ while the field must remain divergence-free and keep the same symmetric gradient in $I_n$. In 2D this is not possible since $I^+_n\setminus (I_n \cup J_n)$ is not connected whenever~$I_n$ has multiple quasi-contact points.
Due to this geometric rigidity in 2D, the above proof is no longer valid: we must abandon the cancellation of the stream function at quasi-contact points and rather consider the extension operator
\[\widetilde Pg:=\Div(w_n\sigma[P_n^1g]).\]
The bound~\eqref{eq:prebnd-Png00} then becomes, for all $r\ge s\ge1$, with $r<\frac{2s}{2-s}$ if $s<2$,
\begin{eqnarray*}
\lefteqn{\|\nabla\widetilde Pg\|_{\Ld^s(I_n^+)}}\\
&\lesssim&\|\nabla P_n^1g\|_{\Ld^s(I_n^+)}+\|\nabla w_n\|_{W^{1,r}(I_n^+)}\Big(\|P_n^1g\|_{\Ld^{\frac{rs}{r-s}}(\R^2)}+\|\sigma[P_n^1g]\|_{W^{1,\frac{rs}{r-s}}(\R^2)}\Big)\\
&\lesssim_{r,s}&\|\nabla P_n^1g\|_{\Ld^s(I_n^+)}+\|\nabla w_n\|_{W^{1,r}(I_n^+)}\|P_n^1g\|_{\Ld^{\frac{rs}{r-s}}(I_n+ B)}
\end{eqnarray*}
where we used the Sobolev embedding, the bound~\eqref{eq:bnd-re-P200} on $\sigma[P_n^1g]$, and Jensen's inequality. Combining this with~\eqref{eq:nabla2wn}, \eqref{eq:first-ext0}, and~\eqref{eq:bnd-re-P2}, we deduce for all $r\ge s\ge1$, with $r<\frac{2s}{2-s}$ if~$s<2$, and for all $p\ge s\vee\frac{2rs}{2(r-s)+rs}$, with $p>2$ if $r=s$,
\begin{equation*}
\|\nabla\widetilde Pg\|_{\Ld^s(I_n^+)}
\,\lesssim_{p,r,s}\,\rho_n^{\frac{3}{2r}-2}\|\D(g)\|_{\Ld^p(I_n)}.
\end{equation*}
Replacing Proposition~\ref{prop:ext-key} by this extension result would lead to corresponding 2D versions of our main results; we skip the detail for shortness.
\end{rem}

\subsection{Proof of Lemma~\ref{lem:Bog2}}
Starting point is the following standard construction based on the Bogovskii operator, e.g.~\cite[Theorem~III.3.1]{Galdi}: given a domain $D$ as in the statement, and given $h\in C_b(D)$ with $\int_{D\setminus\Ic}h=0$, there exists $z^0\in H^1_0(D)^d$ such that
\[\Div(z^0)=h\mathds1_{D\setminus\Ic},\quad\text{in $D$},\]
and for all $1<s<\infty$,
\begin{equation}\label{eq:pre-bnd-Bogo}
\|\nabla z^0\|_{\Ld^s(D)}\,\lesssim_{s}\,K(D)^{d+1}\|h\|_{\Ld^s(D\setminus\Ic)}.
\end{equation}
Next, given $1<s\le r<\infty$, with $r\ne\frac{ds}{d-s}$ if~$s<d$, and with $r<\frac{ds}{d+s-ds}$ if $s<\frac{d}{d-1}$, 
we appeal to the extension operator $P_n$ that we have constructed in Proposition~\ref{prop:ext-key}, and we define
\[\textstyle z\,:=\,z^0-\sum_nP_n(z^0|_{I_n}).\]
By the properties of $P_n$, we find
\[\D(z)|_{\Ic}=0,\qquad\text{and}\qquad\Div(z)=\Div(z^0)=h\mathds1_{D\setminus h},\quad\text{in $D$},\]
and for all $p\ge s\vee\frac{drs}{d(r-s)+rs}$, with $p>d$ if $r=s$,
\begin{eqnarray*}
\|\nabla z\|_{\Ld^s(D)}^s&\lesssim_s&\|\nabla z^0\|_{\Ld^s(D)}^s+\sum_{n:I_n\cap D\ne\varnothing}\|\nabla P_n(z^0|_{I_n})\|_{\Ld^s(I_n^+)}^s\\
&\lesssim_{p,r,s}&\|\nabla z^0\|_{\Ld^s(D)}^s+\sum_{n:I_n\cap D\ne\varnothing}\mu_{r}(\rho_n)^s\|\!\D(z^0)\|_{\Ld^p(I_n)}^s\\
&\lesssim&\Big(|D|+\sum_{n:I_n\cap D\ne\varnothing}\mu_{r}(\rho_n)^\frac{ps}{p-s}\Big)^{1-\frac sp}\|\nabla z^0\|_{\Ld^p(D)}^s.
\end{eqnarray*}
where the last bound follows from Hölder's inequality. Combined with~\eqref{eq:pre-bnd-Bogo}, this yields the conclusion.
\qed

\subsection{Proof of Theorem~\ref{th:extension}}
We split the proof into three steps.

\medskip
\step1 Given $q,S,f$ as in~\eqref{eq:cond-ext}, and given $1<\beta<\infty$ and $\alpha,r$ as in~\eqref{eq:ch-param},
we show that for all~$n$ there exists $z_n\in W^{1,\alpha}(I_n)^{d}$ such that
\begin{equation}\label{eq:constr-rn}
2\int_{I_n}\D(g):\D( z_n)=\int_{I_n^+}g\cdot f-2\int_{I_n^+\setminus I_n}\D(g):q,\quad\forall g\in C^1_c(I_n^+)^d:\Div(g)=0,
\end{equation}
and
\begin{equation}\label{eq:estim-rn}
\|\!\D( z_n)\|_{\Ld^{\alpha}(I_n)}\,\lesssim_{\alpha,\beta,r}\,\mu_{r}(\rho_n)\,\Big(\|f\|_{W^{-1,\beta}(I_n^+)}+\|q\|_{\Ld^\beta(I_n^+\setminus I_n)}\Big).
\end{equation}
While the left-hand side in~\eqref{eq:constr-rn} only involves the restriction $g|_{I_n}\in C^1_b(I_n)^d$ of the test function $g$, the right-hand side involves its extension $g\in C^1_c(I_n^+)^d$. In view of the condition~\eqref{eq:cond-ext}, the choice of the extension does not matter.
Given $1<s\le r<\infty$, with $r\ne\frac{ds}{d-s}$ if~$s<d$, and with $r<\frac{ds}{d+s-ds}$ if $s<\frac{d}{d-1}$,
we recall the extension operator $P_n$ that we have constructed in Proposition~\ref{prop:ext-key}, and the problem~\eqref{eq:constr-rn} then reads
\begin{equation}\label{eq:constr-rn+}
2\int_{I_n}\D(g):\D( z_n)\,=\,\calF_n(g),\qquad
\forall g\in C^1_b(I_n)^d:\Div(g)=0,
\end{equation}
where we have set for abbreviation
\[\calF_n(g)\,:=\,\int_{I_n^+}(P_ng)\cdot f-2\int_{I_n^+\setminus I_n}\D(P_ng):q.\]
By Proposition~\ref{prop:ext-key}, we find for all $p\ge s\vee\frac{drs}{d(r-s)+rs}$, with $p>d$ if $r=s$,
\begin{eqnarray*}
|\calF_n(g)|&\lesssim&\Big(\|f\|_{W^{-1,s'}(I_n^+)}+\|q\|_{\Ld^{s'}(I_n^+\setminus I_n)}\Big)\|\nabla P_ng\|_{\Ld^s(I_n^+)}\nonumber\\
&\lesssim_{p,r,s}&\mu_{r}(\rho_n)\,\Big(\|f\|_{W^{-1,s'}(I_n^+)}+\|q\|_{\Ld^{s'}(I_n^+\setminus I_n)}\Big)\|\!\D(g)\|_{\Ld^p(I_n)}.
\end{eqnarray*}
Appealing to the $\Ld^{p'}$ theory for the Stokes equation, e.g.~\cite[Section~IV.6]{Galdi}, we deduce that there exists a solution $ z_n\in W^{1,p'}(I_n)^d$ of the problem~\eqref{eq:constr-rn+} (unique up to a rigid motion), and that it satisfies
\begin{equation*}
\|\!\D( z_n)\|_{\Ld^{p'}(I_n)}\,\lesssim_{p,r,s}\,\mu_{r}(\rho_n)\Big(\|f\|_{W^{-1,s'}(I_n^+)}+\|q\|_{\Ld^{s'}(I_n^+\setminus I_n)}\Big).
\end{equation*}
Setting $\alpha:=p'$ and $\beta:=s'$, this yields the claim~\eqref{eq:constr-rn}--\eqref{eq:estim-rn}.

\medskip
\step2 Construction of extended flux.\\
Given $1<\beta<\infty$ and $\alpha,r$ as in~\eqref{eq:ch-param},
define
\[\textstyle\tilde q\,:=\,q\mathds1_{\R^d\setminus\Ic}+\sum_{n}\D(z_n)\mathds1_{I_n},\]
with $z_n$ as constructed in Step~1.
Given $g\in C^1_c(\R^d)^d$ with $\Div(g)=0$, we may decompose
\[\textstyle g=g_\circ+\sum_{n}g_n,\qquad g_\circ:=g-\sum_{n}P_ng,\qquad g_n:=P_ng.\]
Using~\eqref{eq:cond-ext} with test function $g_\circ$, and using~\eqref{eq:constr-rn+} with test function $g_n$, we are led to the following integral identity,
\begin{equation}\label{eq:cond-ext-todo1pr}
2\int_{\R^d}\D(g):\tilde q=\int_{\R^d}g\cdot f,\qquad\text{$\forall\,g\in C^1_c(\R^d)^d$: $\Div (g)=0$.}
\end{equation}
Next, we prove the bound~\eqref{eq:ext-est-res} for $\tilde q$.
Given a bounded domain $D\subset\R^d$,
summing~\eqref{eq:estim-rn} over all particles, appealing to Hölder's inequality, and using the Sobolev embedding $\Ld^{d\beta/(d+\beta)}\hookrightarrow W^{-1,\beta}$, we find
\begin{eqnarray}
\lefteqn{\|\tilde q\|_{\Ld^\alpha(D)}^\alpha~\le~\|q\|_{\Ld^\alpha(D\setminus\Ic)}^\alpha+\sum_{n:I_n\cap D\ne\varnothing}\|\!\D(z_n)\|_{\Ld^\alpha(I_n)}^\alpha}\nonumber\\
&\lesssim_{\alpha,\beta,r}&\|q\|_{\Ld^\alpha(D\setminus\Ic)}^\alpha+\sum_{n:I_n\cap D\ne\varnothing}\mu_{r}(\rho_n)^\alpha\Big(\|f\|_{W^{-1,\beta}(I_n^+)}^\beta+\|q\|_{\Ld^\beta(I_n^+\setminus I_n)}^\beta\Big)^\frac\alpha\beta\nonumber\\
&\lesssim&\Big(|D|+\sum_{n:I_n\cap D\ne\varnothing}\mu_{r}(\rho_n)^\frac{\beta\alpha}{\beta-\alpha}\Big)^{1-\frac\alpha\beta}\Big(\|f\|_{\Ld^\frac{d\beta}{d+\beta}(\widehat D)}^\alpha+\|q\|_{\Ld^\beta(\widehat D\setminus\Ic)}^\alpha\Big),\label{eq:ext-est-res0}
\end{eqnarray}
where we recall the notation $\widehat D=D\cup\bigcup_{n:I_n\cap D\ne\varnothing}I_n^+$.

\medskip
\step3 Construction of extended pressure.\\
In view of e.g.~\cite[Proposition~12.10]{JKO94}, the relation~\eqref{eq:cond-ext-todo1pr} for the extension $\tilde q$ ensures the existence of an associated pressure field $\tilde S\in\Ld^1_\loc(\R^d)$, uniquely defined up to a global additive constant, such that
\begin{equation}\label{eq:tildeqE-pres0}
\int_{\R^d}\D(g):\big(2\tilde q-\tilde S\Id\big)=\int_{\R^d}g\cdot f,\qquad\forall g\in C^1_c(\R^d)^d,
\end{equation}
that is, $-\Div(2\tilde q-\tilde S\Id)=f$ in $\R^d$.
It remains to prove the bound~\eqref{eq:ext-est-res} for $\tilde S$.
For all $R\ge1$, by a standard use of the Bogovskii operator, e.g.~\cite[Theorem~III.3.1]{Galdi}, we can construct {$z_R\in W^{1,\alpha'}_0(B_R)^d$} such that
\[\Div(z_R)=\big(T_R|T_R|^{\alpha-2}-\textstyle\fint_{B_R}T_R|T_R|^{\alpha-2}\big)\mathds1_{B_R},\qquad T_R:=\tilde S-\fint_{B_R}\tilde S,\]
and
\[\|\nabla z_R\|_{\Ld^{\alpha'}(B_R)}\,\lesssim_\alpha\,\|T_R|T_R|^{\alpha-2}\|_{\Ld^{\alpha'}(B_R)}\,\lesssim\,\|\tilde S-\textstyle\fint_{B_R}\tilde S\|_{\Ld^{\alpha}(B_R)}^{\alpha-1}.\]
Testing~\eqref{eq:tildeqE-pres0} with $g=z_R$, we find
\[\int_{\R^d}\tilde S\,\Div(z_R)=2\int_{\R^d}\D(z_R):\tilde q-\int_{\R^d}z_R\cdot f,\]
and thus, using the properties of $z_R$,
\[\|\tilde S-\textstyle\fint_{B_R}\tilde S\|_{\Ld^{\alpha}(B_R)}
\,\lesssim\,\|f\|_{W^{-1,\alpha}(B_R)}+\|\tilde q\|_{\Ld^\alpha(B_R)}.\]
Combined with~\eqref{eq:ext-est-res0}, this yields the conclusion~\eqref{eq:ext-est-res}.\qed

\section{Homogenization}
This section is devoted to the proof of Theorems~\ref{prop:cor} and~\ref{th:main-qsh}.
While Tartar's oscillating test function method as used in~\cite{DG-19} is not quite appropriate to the present setting without uniform separation, we provide an alternative argument based on div-curl ideas combined with the extension result in Theorem~\ref{th:extension}, as inspired by the work of Jikov~\cite{Jikov-87,Jikov-90} on homogenization problems with stiff inclusions (see also~\cite[Section~3.2]{JKO94}).

\subsection{Construction of correctors}
We start with the proof of Theorem~\ref{prop:cor},
which we shall deduce from our results in~\cite{DG-19} for uniformly separated particles, via an approximation argument together with suitable a priori estimates. The improved pressure estimates in Proposition~\ref{prop:Sig-int} are deduced simultaneously.

\begin{proof}[Proof of Theorem~\ref{prop:cor} and Proposition~\ref{prop:Sig-int}]
Given $2\le r\ne\frac{2d}{d-2}$ and $1\le\alpha\le2\wedge\frac{2dr}{r(d-2)+2d}$, with $\alpha<\frac{d}{d-1}$ if $r=2$, we assume that interparticle distances satisfy
\begin{equation}\label{eq:mom-1}
\textstyle\sum_n\expecm{\mu_{r}(\rho_n)^\frac{2\alpha}{2-\alpha}\mathds1_{0\in I_n}}\,<\,\infty,
\end{equation}
and we shall then prove Theorem~\ref{prop:cor}, with pressure $\Sigma_E$ in $\Ld^\alpha(\Omega)$ provided $\alpha>1$. Optimizing in $r$ further yields Proposition~\ref{prop:Sig-int}.
We split the proof into two main steps.

\medskip
\step1 Approximations with uniformly separated particles.\\
For $0<\kappa\le\frac\delta2$, we consider the restricted inclusions
\[I_n^\kappa:=\{x\in I_n:\dist(x,\partial I_n)>\kappa\},\qquad\Ic^\kappa:=\bigcup_nI_n^\kappa,\]
which still satisfy Assumptions~\ref{Hd} and~\ref{Hd'} with $\delta$ replaced by $\frac\delta2$ and with minimal interparticle distance~$\rho_n^\circ\ge\kappa$, cf.~\eqref{eq:def-rn0}, that is, $(I_n^\kappa+\kappa B)\cap(I_m^\kappa+\kappa B)=\varnothing$ for all $n\ne m$. In this context with uniformly separated particles, we may apply~\cite[Proposition~2.1]{DG-19}, which ensures the existence and uniqueness of a corrector~$\psi_E^\kappa$ and of an associated pressure~$\Sigma_E^\kappa$ that satisfy the different properties stated in Theorem~\ref{prop:cor} with $\Ic$ replaced by $\Ic^\kappa$.
In addition, we show that the following moment bounds hold uniformly with respect to the parameter~$\kappa>0$: for $\alpha$ as in the moment condition~\eqref{eq:mom-1},
\begin{eqnarray}
\expecm{|\nabla\psi_E^\kappa|^2}&\lesssim&|E|^2,\label{eq:bnd-unif-eta-1}\\
\expecm{|\Sigma_E^\kappa|^\alpha\mathds1_{\R^d\setminus\Ic^\kappa}}&\lesssim&|E|^\alpha.\label{eq:bnd-unif-eta-2}
\end{eqnarray}
These two estimates are established in the following two substeps.

\medskip
\substep{1.1} Proof of~\eqref{eq:bnd-unif-eta-1}.\\
In terms of the extension operator $P_n$ that we have constructed in Proposition~\ref{prop:ext-key}, we consider the following stationary random vector field,
\[\textstyle\phi_E^\circ\,:=\,-\sum_nP_n(E(x-x_n)),\]
and we show that it satisfies
\begin{equation}\label{eq:prop-psi0}
(\D(\phi^\circ_E)+E)|_{\Ic}=0,\quad~\Div(\phi^\circ_E)=0,\quad~\expec{|\!\D(\phi_E^\circ)|^2}\lesssim|E|^2,\quad~\expec{\D(\phi_E^\circ)}=0.
\end{equation}
The first two properties follow from the construction of $P_n$ with $\D(P_n(E(x-x_n)))|_{I_n}=E$ and $\Div(P_n(E(x-x_n)))=0$.
Next, stationarity allows to write
\[\expec{|\!\D(\phi_E^\circ)|^2}\,=\,\expecM{\fint_{B}|\!\D(\phi_E^\circ)|^2}\,=\,\expecM{\tfrac1{|B|}\sum_{n}\int_{I_n^+\cap B}|\!\D(P_n(E(x-x_n)))|^2},\]
hence, by Proposition~\ref{prop:ext-key},
\[\expec{|\!\D(\phi_E^\circ)|^2}\,\lesssim\,|E|^2\,\expecM{\tfrac1{|B|}\sum_{n:I_n^+\cap B\ne\varnothing}\mu_2(\rho_n)^2}.\]
This can then be estimated as follows, by stationarity,
\[\expec{|\!\D(\phi_E^\circ)|^2}\,\lesssim\,|E|^2\,\expecM{\tfrac1{|B|}\int_{B_{3}}\sum_{n}\mu_2(\rho_n)^2\mathds1_{I_n}}\,\lesssim\,|E|^2\,\expecM{\sum_{n}\mu_2(\rho_n)^2\mathds1_{0\in I_n}},\]
so that the third property in~\eqref{eq:prop-psi0} follows from the moment assumption~\eqref{eq:mom-1} with $r=2$ and $\alpha=1$.
Finally, stationarity allows to write for all $R>0$,
\[\expec{\D(\phi_E^\circ)}\,=\,\expecM{\fint_{B_R}\D(\phi_E^\circ)},\]
hence, inserting the definition of $\D(\phi_E^\circ)\in\Ld^2(\Omega)$, and letting $R\uparrow\infty$ to neglect boundary terms,
\[\expec{\D(\phi_E^\circ)}\,=\,-\lim_{R\uparrow\infty}\expecM{|B_R|^{-1}\sum_{n:I_n^+\subset B_R}\int_{I_n^+}\D(P_n(E(x-x_n)))}.\]
Combined with the observation that $\int_{I_n^+}\D(P_n(E(x-x_n)))=0$, this concludes the proof of the last property in~\eqref{eq:prop-psi0}.

\medskip\noindent
With this construction at hand, and noting that $\Ic^\kappa\subset\Ic$,
testing the variational problem~\eqref{eq:var-psi} for $\psi_E^\kappa$ with the test function $\phi_E^\circ$ yields
\begin{equation}\label{eq:bnd-Dpsieta}
\expec{|\!\D(\psi_E^\kappa)+E|^2}\,\le\,\expec{|\!\D(\phi_E^\circ)+E|^2}\,\lesssim\,|E|^2.
\end{equation}
It remains to turn this into an a priori estimate on the full gradient $\nabla\psi_E^\kappa$. For that purpose, we decompose
\begin{equation}\label{eq:decomp-Dnabpsi}
|\nabla\psi_E^\kappa|^2=2|\!\D(\psi_E^\kappa)|^2-\nabla_j(\psi_E^\kappa)_i\nabla_i(\psi_E^\kappa)_j.
\end{equation}
For all $R\ge1$, choose a smooth averaging function $\chi_R\in C^\infty_c(\R^d;\R^+)$ such that $\chi_R$ is constant in $B_R$, vanishes outside $B_{2R}$, and satisfies $\int_{\R^d}\chi_R=1$ and $|\nabla\chi_R|\lesssim R^{-d-1}$.
An integration by parts together with the constraint $\Div(\psi_E^\kappa)=0$ yields
\[\int_{\R^d}\chi_R\nabla_j(\psi_E^\kappa)_i\nabla_i(\psi_E^\kappa)_j\,=\,-\int_{\R^d}(\nabla\chi_R\otimes \psi_E^\kappa) :\nabla\psi_E^\kappa,\]
and thus, by definition of $\chi_R$ and by scaling,
\[\Big|\int_{\R^d}\chi_R\nabla_j(\psi_E^\kappa)_i\nabla_i(\psi_E^\kappa)_j\Big|\,\lesssim\,\|R^{-1}\psi_E^\kappa(R\cdot)\|_{\Ld^2(B_2)}\|\nabla\psi_E^\kappa(R\cdot)\|_{\Ld^2(B_2)}.\]
Passing to the limit $R\uparrow\infty$ and appealing to the ergodic theorem, in view of the stationarity of $\nabla\psi_E^\kappa$ and the sublinearity of $\psi_E^\kappa$, cf.~\eqref{eq:conv-cor}, we deduce $\expec{\nabla_j(\psi_E^\kappa)_i\nabla_i(\psi_E^\kappa)_j}=0$, so that the decomposition~\eqref{eq:decomp-Dnabpsi} entails $\expec{|\nabla\psi_E^\kappa|^2}=2\,\expec{|\!\D(\psi_E^\kappa)|^2}$ and the bound~\eqref{eq:bnd-Dpsieta} yields the claim~\eqref{eq:bnd-unif-eta-1}.

\medskip
\substep{1.2} Proof of~\eqref{eq:bnd-unif-eta-2}.\\
We appeal to Lemma~\ref{lem:Bog2} in the following form (with $s=2$ and $p=\alpha'$, with $\alpha,r$ as in the moment condition~\eqref{eq:mom-1}): there exists $z_R^\kappa\in H^1_0(B_R)^d$ such that $\D(z_R^\kappa)|_{\Ic^\kappa}=0$ and
\[\Div(z_R^\kappa)\,=\,\Big(T_R^\kappa|T_R^\kappa|^{\alpha-2}-\textstyle\fint_{B_R\setminus\Ic^\kappa}T_R^\kappa|T_R^\kappa|^{\alpha-2}\Big)\mathds1_{B_R\setminus\Ic^\kappa},\qquad T_R^\kappa\,:=\,\Sigma_E^\kappa-\fint_{B_R\setminus\Ic^\kappa}\Sigma_E^\kappa,\]
and such that
\begin{eqnarray*}
\|\nabla z_R^\kappa\|_{\Ld^2(B_R)}&\lesssim_{\alpha,r}&\Lambda(B_R;r,\tfrac{2\alpha}{2-\alpha})\,\|T_R^\kappa|T_R^\kappa|^{\alpha-2}\|_{\Ld^{\alpha'}(B_R\setminus\Ic^\kappa)}\\
&\lesssim&\Lambda(B_R;r,\tfrac{2\alpha}{2-\alpha})\,\|\textstyle\Sigma_E^\kappa-\fint_{B_R\setminus\Ic^\kappa}\Sigma_E^\kappa\|_{\Ld^{\alpha}(B_R\setminus\Ic^\kappa)}^{\alpha-1}.
\end{eqnarray*}
Testing the corrector equation~\eqref{eq:cor} for $(\psi_E^\kappa,\Sigma_E^\kappa)$ with this test function $z_R^\kappa$, we find
\[\int_{B_R}\Sigma_E^\kappa\,\Div(z_R^\kappa)\,=\,2\int_{B_R}\D(z_R^\kappa):\D(\psi_E^\kappa),\]
and thus, using the above properties of $z_R^\kappa$,
\begin{equation}\label{eq:pres-estimate}
\|\Sigma_E^\kappa-\textstyle\fint_{B_R\setminus\Ic^\kappa}\Sigma_E^\kappa\|_{\Ld^\alpha(B_R\setminus\Ic^\kappa)}\,\lesssim_{\alpha,r}\,\Lambda(B_R;r,\tfrac{2\alpha}{2-\alpha})\,\|\!\D(\psi_E^\kappa)\|_{\Ld^2(B_R)}.
\end{equation}
Dividing both sides by $R^{d/\alpha}$, recalling the definition~\eqref{eq:lambda}--\eqref{eq:def-mu} of $\Lambda$, passing to the limit $R\uparrow\infty$, and appealing to the ergodic theorem, recalling that $\nabla\psi_E^\kappa$ and~$\Sigma_E^\kappa$ are stationary with vanishing expectation, and using~\eqref{eq:bnd-unif-eta-1}, we deduce
\begin{equation*}
\expecm{|\Sigma_E^\kappa|^{\alpha}\mathds1_{\R^d\setminus\Ic^\kappa}}^\frac1\alpha\,\lesssim\,|E|\Big(1+\textstyle\sum_n\expecm{\mu_{r}(\rho_n)^{\frac{2\alpha}{2-\alpha}}\mathds1_{0\in I_n}}\Big)^{\frac{2-\alpha}{2\alpha}}.
\end{equation*}
and the claim~\eqref{eq:bnd-unif-eta-2} follows from the moment assumption~\eqref{eq:mom-1}.

\medskip
\step2 Conclusion.\\
In view of the uniform bounds~\eqref{eq:bnd-unif-eta-1} and~\eqref{eq:bnd-unif-eta-2}, provided $\alpha>1$, we may consider some weak limit point $(\nabla\psi_E,\Sigma_E\mathds1_{\R^d\setminus\Ic})$ of $\{(\nabla\psi_E^\kappa,\Sigma_E^\kappa\mathds1_{\R^d\setminus\Ic^\kappa})\}_{\kappa>0}$ in $\Ld^2(\Omega)^{d\times d}\times\Ld^\alpha(\Omega)$ as $\kappa\downarrow0$. It follows that $\nabla\psi_E$ is stationary with vanishing expectation and finite second moments, that it satisfies $\Div(\psi_E)=0$ and $(\D(\psi_E)+E)|_\Ic=0$,
and that $\D(\psi_E)$ is the unique solution of the limiting variational problem~\eqref{eq:var-psi}.
Moreover, passing to the limit in the weak formulation of~\eqref{eq:cor}, we find
\begin{equation}\label{eq:weak-form-cor-pres}
2\int_{\R^d}\D(g):\D(\psi_E)\,=\,\int_{\R^d}\Sigma_E\,\Div(g),\qquad\forall g\in C^1_c(\R^d)^d:\D(g)|_{\Ic}=0,
\end{equation}
hence, in particular,
\begin{equation}\label{eq:strongformstokes}
-\triangle\psi_E+\nabla\Sigma_E\,=\,0,\qquad\text{in $\R^d\setminus\Ic$}.
\end{equation}
The pressure field $\Sigma_E$ in this equation is uniquely defined up to a global constant in view of the almost sure connectedness of $\R^d\setminus\Ic$, and is thus fully determined by the condition $\expecm{\Sigma_E\mathds1_{\R^d\setminus\Ic}}=0$.
In addition, in view of the regularity of the particle boundaries, cf.~Assumption~\ref{Hd}, the regularity theory for the Stokes equation (e.g.~\cite[Section~IV]{Galdi}) entails that $(\psi_E,\Sigma_E)$ is $C^2$ smooth in $\R^d\setminus\Ic$ up to the boundary, and equation~\eqref{eq:strongformstokes} is thus satisfied in the strong sense.
Next, for all $n$, for all $V\in\R^d$ and~$\Theta\in\Md^\Skew$, in terms of the cut-off function $w_n$ that we have constructed in Lemma~\ref{lem:wn}, we may test equation~\eqref{eq:weak-form-cor-pres} with $g=w_n(V+\Theta (x-x_n))\in W^{1,\alpha'}_0(I_n^+)$, which indeed satisfies $\D(g)|_\Ic=0$, and an integration by parts then yields
\begin{eqnarray*}
0&=&\int_{\R^d}\D\big(w_n(V+\Theta (x-x_n))\big):\sigma(\psi_E+Ex,\Sigma_E)\\
&=&-\int_{\partial I_n}(V+\Theta (x-x_n))\cdot\sigma(\psi_E+Ex,\Sigma_E)\nu,
\end{eqnarray*}
showing that the boundary conditions in~\eqref{eq:cor} are almost surely satisfied in a pointwise sense.
Finally, the weak convergence of $(\nabla\psi_E,\Sigma_E\mathds1_{\R^d\setminus\Ic})(\tfrac\cdot\e)$ to $0$ in~\eqref{eq:conv-cor} follows from $\expecm{(\nabla\psi_E,\Sigma_E\mathds1_{\R^d\setminus\Ic})}=0$ by the ergodic theorem, while the sublinearity of~$\psi_E$ in form of the strong convergence of $\e\psi_E(\tfrac\cdot\e)$ to $0$ is a standard result for random fields with stationary gradient having vanishing expectation, e.g.~\cite[Section~7]{JKO94}.
\end{proof}

\subsection{Extension of fluxes}
Applying Theorem~\ref{th:extension} to the corrector $\psi_E$, cf.~\eqref{eq:cor}, and to the solution $u_\e$ of the Stokes problem~\eqref{eq:st-het},
we obtain the following useful extension result for the fluxes
\[q_E:=\D(\psi_E)+E,\qquad p_\e:=\D(u_\e).\]

\begin{cor}\label{cor:extension}
On top of Assumptions~\ref{Hd} and~\ref{Hd'}, given a bounded Lipschitz domain~$U\subset\R^d$, and given $2\le r\ne\frac{2d}{d-2}$ and $1<\alpha\le2\wedge\frac{2dr}{r(d-2)+2d}$, with $\alpha<\frac{d}{d-1}$ if $r=2$, assume that interparticle distances satisfy the following moment condition, almost surely,
\begin{equation}\label{eq:mom-1-rep}
\textstyle\limsup_{\e\downarrow0} \e^d\sum_{n\in\Nc_\e(U)}\mu_{r}(\rho_{n;U,\e})^\frac{2\alpha}{2-\alpha}\,<\,\infty.
\end{equation}
Then the following properties hold.
\begin{enumerate}[(i)]
\item For all $E\in\Md_0^\Sym$, there exist a stationary element $\tilde q_E\in\Ld^\alpha(\Omega;\Ld^\alpha_\loc(\R^d)^{d\times d}_\Sym)$ with $\Tr(\tilde q_E)=0$, and an associated stationary pressure field $\tilde\Sigma_E\in\Ld^\alpha(\Omega;\Ld^\alpha_\loc(\R^d))$, such that almost surely
\begin{gather}
\qquad(\tilde q_E,\tilde \Sigma_E)|_{\R^d\setminus\Ic}=(q_E,\Sigma_E)|_{\R^d\setminus\Ic},\nonumber\\
\qquad-\Div(2\tilde q_E-\tilde \Sigma_E\Id)=0,\qquad\text{in $\R^d$},\label{eq:psiSig-ext-eqn}
\end{gather}
and
\[\qquad\|\tilde q_E\|_{\Ld^\alpha(\Omega)}+\|\tilde \Sigma_E-\expecm{\tilde\Sigma_E}\|_{\Ld^\alpha(\Omega)}\,\lesssim_{\alpha,r}\,|E|.\]
\item There exists $\tilde p_\e\in\Ld^\alpha(\Omega;\Ld^\alpha(U)^{d\times d}_\Sym)$ with $\Tr(\tilde p_\e)=0$, and an associated pressure field $\tilde S_\e\in\Ld^\alpha(\Omega;\Ld^\alpha(U))$, such that almost surely
\begin{gather}
\qquad(\tilde p_\e,\tilde S_\e)|_{U\setminus\Ic_\e(U)}=(p_\e,S_\e)|_{U\setminus\Ic_\e(U)},\nonumber\\
\qquad-\Div(2\tilde p_\e-\tilde S_\e\Id)=f\mathds1_{U\setminus\Ic_\e(U)},\qquad\text{in $U$},\label{eq:pS-ext-eqn}
\end{gather}
and
\[\qquad\limsup_{\e\downarrow0}\big(\|\tilde p_\e\|_{\Ld^\alpha(U)}+\|\tilde S_\e-\textstyle\fint_U\tilde S_\e\|_{\Ld^\alpha(U)}\big)\,\lesssim_{U,\alpha,r}\,\|f\|_{\Ld^\frac{2d}{d+2}(U)}.\qedhere\]
\end{enumerate}
\end{cor}

\begin{proof}
We split the proof into two steps.

\medskip
\step1 Proof of~(i)\\
The corrector equation~\eqref{eq:cor} ensures that the flux $q_E=\D(\psi_E)+E\in\Ld^2(\Omega;\Ld^2_\loc(\R^d)^{d\times d}_\Sym)$ satisfies \mbox{$\Tr(q_E)=0$} and
\[\int_{\R^d}\D(g):q_E=0,\qquad\forall g\in C^1_c(\R^d)^d:\Div(g)=0,\,\D(g)|_{\Ic}=0.\]
Given $\alpha,r$ as in~\eqref{eq:mom-1-rep}, Theorem~\ref{th:extension} provides an extension $\tilde q_E\in\Ld^\alpha(\Omega;\Ld^\alpha_\loc(\R^d)^{d\times d}_\Sym)$ with $\Tr(\tilde q_E)=0$, and an associated pressure field $\tilde\Sigma_E\in\Ld^\alpha(\Omega;\Ld^\alpha_\loc(\R^d))$, such that
\begin{equation}\label{eq:qE-ext-eqn}
\tilde q_E|_{\R^d\setminus\Ic}=q_E|_{\R^d\setminus\Ic},\qquad\text{and}\qquad-\Div(2\tilde q_E-\tilde \Sigma_E\Id)=0,\quad\text{in $\R^d$},
\end{equation}
and such that the following estimate holds, for all $R\ge1$,
\begin{equation}\label{eq:ext-est-res-use}
\|\tilde q_E\|_{\Ld^\alpha(B_R)}+\|\tilde \Sigma_E-\textstyle\fint_{B_R}\tilde \Sigma_E\|_{\Ld^{\alpha}(B_R)}
\,\lesssim_{\alpha,r}\,\Lambda(B_R;r,\tfrac{2\alpha}{2-\alpha})\|q_E\|_{\Ld^2(\widehat B_R\setminus\Ic)}.
\end{equation}
In addition, the construction in the proof of Theorem~\ref{th:extension} ensures that $\tilde q_E$ can be chosen stationary.
Since $\tilde q_E$ coincides with $q_E=\D(\psi_E)+E$ on $\R^d\setminus\Ic$, we deduce from~\eqref{eq:qE-ext-eqn}, in particular,
\[-\triangle\psi_E+\nabla\tilde \Sigma_E=0,\qquad\text{in $\R^d\setminus\Ic$}.\]
In view of~\eqref{eq:cor}, recalling that $\R^d\setminus\Ic$ is almost surely connected, we deduce that the pressure $\tilde\Sigma_E$ must coincide with $\Sigma_E$ in $\R^d\setminus\Ic$ up to a global constant. Therefore, $\tilde\Sigma_E$ is uniquely determined for instance by the choice $\tilde\Sigma_E|_{\R^d\setminus\Ic}=\Sigma_E|_{\R^d\setminus\Ic}$. For this choice, as $\tilde q_E$ and $\Sigma_E\mathds1_{\R^d\setminus\Ic}$ are stationary, uniqueness entails that~$\tilde \Sigma_E$ is also stationary.

\medskip\noindent
Dividing both sides of~\eqref{eq:ext-est-res-use} by $R^{d/\alpha}$, recalling the definition~\eqref{eq:lambda} of $\Lambda$, passing to the limit $R\uparrow\infty$, appealing to the ergodic theorem, in view of the stationarity of $\tilde q_E,\tilde\Sigma_E$, and using the energy bound $\expec{|q_E|^2}\lesssim|E|^2$, we obtain
\begin{equation*}
\|\tilde q_E\|_{\Ld^\alpha(\Omega)}+\|\tilde\Sigma_E-\expecm{\tilde\Sigma_E}\|_{\Ld^{\alpha}(\Omega)}
\,\lesssim_{\alpha,r}\,|E|\Big(1+\sum_n\expecm{\mu_{r}(\rho_n)^{\frac{2\alpha}{2-\alpha}}\mathds1_{0\in I_n}}\Big)^{\frac{2-\alpha}{2\alpha}}.
\end{equation*}
Combined with the moment condition~\eqref{eq:mom-1-rep}, with $\mu_{r}(\rho_n)\le\mu_{r}(\rho_{n;U,\e})$, this yields the conclusion.

\medskip
\step2 Proof of~(ii).\\
Equation~\eqref{eq:st-het} ensures that the flux $p_\e=\D(u_\e)\in\Ld^2(\Omega;\Ld^2(U)^{d\times d}_\Sym)$ satisfies $\Tr(p_\e)=0$ and
\[2\int_{U}\D(g):p_\e=\int_{U\setminus\Ic_\e(U)}g\cdot f,\qquad\forall g\in C^1_c(U)^d:\Div(g)=0,\,\D(g)|_{\Ic_\e(U)}=0.\]
Given $\alpha,r$ as in~\eqref{eq:mom-1-rep}, Theorem~\ref{th:extension} provides an extension $\tilde p_\e\in\Ld^\alpha(\Omega;\Ld^\alpha(U)^{d\times d}_\Sym)$ with \mbox{$\Tr(\tilde p_\e)=0$}, and an associated pressure field $\tilde S_\e\in\Ld^\alpha(\Omega;\Ld^\alpha(U))$, such that
\[\tilde p_\e|_{U\setminus\Ic_\e(U)}=p_\e|_{U\setminus\Ic_\e(U)},\qquad\text{and}\qquad-\Div(2\tilde p_\e-\tilde S_\e\Id)=f\mathds1_{U\setminus\Ic_\e(U)},\quad\text{in $U$},\]
and by scaling, the estimate~\eqref{eq:ext-est-res} takes on the following guise,
\begin{multline*}
\|\tilde p_\e\|_{\Ld^\alpha(U)}+\|\tilde S_\e-\textstyle\fint_{U}\tilde S_\e\|_{\Ld^\alpha(U)}\\
\,\lesssim_{U,\alpha,r}\,\Big(|U|+\e^d\sum_{n\in\Nc_\e(U)}\mu_{r}(\rho_{n;U,\e})^{\frac{2\alpha}{2-\alpha}}\Big)^{\frac{2-\alpha}{2\alpha}}\Big(\|f\|_{\Ld^\frac{2d}{d+2}(U)}+\|p_\e\|_{\Ld^2(U\setminus\Ic_\e(U))}\Big).
\end{multline*}
Combined with the energy bound $\|p_\e\|_{\Ld^2(U\setminus\Ic_\e(U))}\lesssim\|f\|_{\Ld^{2d/(d+2)}(U)}$, and with the moment condition~\eqref{eq:mom-1-rep}, this yields the conclusion.
\end{proof}

\begin{rem}\label{rem:extension}
In view of the construction in the proof of Theorem~\ref{th:extension}, it is easily checked that the above-constructed extended fluxes $\tilde q_E$, $\tilde p_\e$ can be viewed as limiting fluxes for corresponding Stokes problems with a suspension of droplets with diverging shear viscosity.
More precisely, for all $\kappa>0$, we consider the following corrector problem
\[\left\{\begin{array}{ll}
-\Div\big(2(\mathds1_{\R^d\setminus\Ic}+\kappa\mathds1_{\Ic})(\D(\psi_E^\kappa)+E)\big)+\nabla\Sigma_E^\kappa=0,&\text{in $\R^d$,}\\
\Div(\psi_E^\kappa)=0,&\text{in $\R^d$.}
\end{array}\right.\]
Under the assumptions of Corollary~\ref{cor:extension}, in the limit $\kappa\uparrow\infty$, there holds $\D(\psi_E^\kappa)\cvf\D(\psi_E)$ in $\Ld^2(\Omega)$ and corresponding fluxes converge,
\[2(\mathds1_{\R^d\setminus \Ic}+\kappa\mathds1_{\Ic})(\D(\psi_E^\kappa)+E)-\Sigma_E^\kappa\Id~~\cvf~~\tilde q_E,\qquad\text{in $\Ld^\alpha(\Omega)$},\]
and a similar result holds for $\tilde p_\e$. We skip the detail for shortness.
\end{rem}

Next, we compute $\expecm{\tilde q_E}$ and $\expecm{\tilde\Sigma_E}$, which happen to provide alternative definitions of the effective constants $\Bb,\bb$. Note in particular that these ensemble averages do not depend on the actual choice of the extension $\tilde q_E$ in Corollary~\ref{cor:extension}(i).

\begin{lem}[Effective constants]\label{lem:Bb-ext}
On top of Assumptions~\ref{Hd} and~\ref{Hd'},
let $(\tilde q_E,\tilde\Sigma_E)$ be defined as in Corollary~\ref{cor:extension}(i) for some $\alpha>1$.
Then we have almost surely, as~$\e\downarrow0$,
\begin{equation}\label{eq:weak-tilde-qS}
\tilde q_E(\tfrac\cdot\e)\cvf\expec{\tilde q_E}=\Bb E,\qquad
\tilde \Sigma_E(\tfrac\cdot\e)\cvf\expecm{\tilde\Sigma_E}=-\bb:E,\qquad\text{weakly in $\Ld^\alpha_\loc(\R^d)$}.
\end{equation}
In addition, provided $\alpha\ge\frac{2d}{d+2}$, these convergences are almost surely strong in $H^{-1}_\loc(\R^d)$.
\end{lem}

\begin{proof}
We split the proof into two steps.

\medskip
\step1 Proof of weak convergences~\eqref{eq:weak-tilde-qS}.\\
As $\tilde q_E$ and $\tilde\Sigma_E$ are stationary, the ergodic theorem implies almost surely the weak convergences $\tilde q_E(\tfrac\cdot\e)\cvf\expec{\tilde q_E}$ and $\tilde\Sigma_E(\tfrac\cdot\e)\cvf\E[\tilde\Sigma_E]$ in $\Ld^\alpha_\loc(\R^d)$, and it remains to compute these two expectations.
For that purpose, up to an approximation argument as in the proof of Theorem~\ref{prop:cor}, we may assume without loss of generality $\alpha>\frac{2d}{d+2}$.
We split the proof into two further substeps.

\medskip
\substep{1.1} Proof that $\Bb E=\expec{\tilde q_E}$.\\
For all $R\ge1$, we set $\chi_R:=R^{-d}\chi(\frac1R\cdot)$, for some smooth averaging function $\chi\in C^\infty_c(\R^d;\R^+)$ such that $\chi$ is constant in $B$, vanishes outside $B_2$, and satisfies $\int_{\R^d}\chi=1$.
Given~$E'\in\Md_0^\Sym$, as $q_E=\D(\psi_E)+E$ is stationary, the definition~\eqref{eq:def-Bb} of $\Bb$ and the ergodic theorem yield almost surely,
\begin{equation}\label{eq:rewr-Bb-1}
E':\Bb E\,=\,\expec{q_{E'}:q_E}\,=\,\lim_{R\uparrow\infty}\int_{\R^d}\chi_R\,q_{E'}:q_E.
\end{equation}
Since $q_{E'}$ vanishes in $\Ic$, cf.~\eqref{eq:cor}, and since $q_E$ coincides with its extension $\tilde q_E$ in $\R^d\setminus\Ic$, we find
\[q_{E'}:q_E\,=\,q_{E'}:\tilde q_E\,=\,E':\tilde q_E+\D(\psi_{E'}):\tilde q_E.\]
Inserting this identity into~\eqref{eq:rewr-Bb-1}, and noting that the ergodic theorem implies the almost sure convergence $\int_{\R^d}\chi_R\,\tilde q_E\to\expec{\tilde q_E}$, we find
\begin{equation}\label{eq:comput-Bb-tildeq-pr}
E':\Bb E\,=\,E':\expec{\tilde q_E}+\lim_{R\uparrow\infty}\int_{\R^d}\chi_R\,\D(\psi_{E'}):\tilde q_E.
\end{equation}
In order to prove the claim $\Bb E=\expec{\tilde q_E}$, it remains to show that the last limit vanishes,
\begin{equation}\label{eq:comput-Bb-tildeq-pr-todo}
\lim_{R\uparrow\infty}\int_{\R^d}\chi_R\,\D(\psi_{E'}):\tilde q_E\,=\,0.
\end{equation}
Integrating by parts, using the properties~\eqref{eq:psiSig-ext-eqn} of the extensions $(\tilde q_E,\tilde\Sigma_E)$, and using the constraint $\Div(\psi_{E'})=0$, we find
\begin{eqnarray*}
\int_{\R^d}\chi_R\,\D(\psi_{E'}):\tilde q_E
&=&\int_{\R^d}\D(\chi_R\psi_{E'}):\tilde q_E-\int_{\R^d}(\nabla\chi_R\otimes\psi_{E'}):\tilde q_E\\
&=&\frac12\int_{\R^d}\tilde \Sigma_E\,\Div(\chi_R\psi_{E'})-\int_{\R^d}(\nabla\chi_R\otimes\psi_{E'}):\tilde q_E\\
&=&-\frac12\int_{\R^d}(\nabla\chi_R\otimes\psi_{E'}):\big(2\tilde q_E-\tilde\Sigma_E\Id\big).
\end{eqnarray*}
The relation $\Div(\psi_{E'})=0$ entails $\int_{\R^d}\nabla\chi_R\cdot\psi_{E'}=0$, which allows to add any constant to the pressure $\tilde \Sigma_E$ in the right-hand side. In view of the properties of the averaging function~$\chi_R$, Hölder's inequality leads to
\begin{eqnarray}\label{eq:bnd-ipp-qStil}
\lefteqn{\hspace{-0.8cm}\Big|\int_{\R^d}\chi_R\,\D(\psi_{E'}):\tilde q_E\Big|
\,\lesssim\,\int_{B_{2}}|\tfrac1R\psi_{E'}(R\cdot)|\Big(|\tilde q_E(R\cdot)|+\Big|\tilde\Sigma_E(R\cdot)-\fint_{B_{2}}\tilde\Sigma_E(R\cdot)\Big|\Big)}\\
&\lesssim&\|\tfrac1R\psi_{E'}(R\cdot)\|_{\Ld^{\alpha'}(B_2)}\Big(\|\tilde q_E(R\cdot)\|_{\Ld^\alpha(B_2)}+\Big\|\tilde\Sigma_E(R\cdot)-\fint_{B_2}\tilde\Sigma_E(R\cdot)\Big\|_{\Ld^\alpha(B_2)}\Big).\nonumber
\end{eqnarray}
As the choice $\alpha>\frac{2d}{d+2}$ entails $\alpha'<\frac{2d}{d-2}$, we can use the
sublinearity of $\psi_{E'}$ in~$\Ld^{\alpha'}$, cf.~\eqref{eq:conv-cor}, together with the boundedness of $\{(\tilde q_E,\tilde\Sigma_E(R\cdot)-\fint_{B_2}\tilde\Sigma_E(R\cdot))\}_R$ in~$\Ld^\alpha(B_2)$, cf.~Corollary~\ref{cor:extension}(i),
and the claim~\eqref{eq:comput-Bb-tildeq-pr-todo} follows.

\medskip
\substep{1.2} Proof that $\bb:E=-\expecm{\tilde\Sigma_E}$.\\
In terms of the cut-off function $w_n$ that we have constructed in Lemma~\ref{lem:wn}, integrating by parts, and recalling that the corrector equation~\eqref{eq:cor} yields $\Div(\sigma(\psi_E+Ex,\Sigma_E))=0$ in $I_n^+\setminus I_n$,
the definition~\eqref{eq:def-bb} of $\bb$ becomes
\begin{eqnarray*}
\bb:E&=&\frac1d\,\expecM{\sum_n\frac{\mathds1_{I_n}}{|I_n|}\int_{\partial I_n}(x-x_n)\cdot\sigma(\psi_E+Ex,\Sigma_E)\nu}\\
&=&-\frac1d\,\expecM{\sum_n\frac{\mathds1_{I_n}}{|I_n|}\int_{I_n^+\setminus I_n}\Div\big(w_n\,\sigma(\psi_E+Ex,\Sigma_E)\,(x-x_n)\big)}\\
&=&-\frac1d\,\expecM{\sum_n\frac{\mathds1_{I_n}}{|I_n|}\int_{I_n^+\setminus I_n}\D\big((x-x_n)w_n\big):\sigma(\psi_E+Ex,\Sigma_E)}.
\end{eqnarray*}
Writing $\sigma(\psi_E+Ex,\Sigma_E)=2q_E-\Sigma_E\Id$ in $I_n^+\setminus I_n$, and using the extensions $\tilde q_E$ and $\tilde\Sigma_E$ as in~\eqref{eq:psiSig-ext-eqn}, we are led to
\begin{equation*}
\bb:E\,=\,\frac1d\,\expecM{\sum_n\frac{\mathds1_{I_n}}{|I_n|}\int_{I_n}\D\big((x-x_n)w_n\big):\big(2\tilde q_E-\tilde\Sigma_E\Id\big)}.
\end{equation*}
Since $\D((x-x_n)w_n)=\Id$ in $I_n$ and since $\Tr(\tilde q_E)=0$, we deduce
\begin{equation*}
\bb:E\,=\,-\expecM{\sum_n\frac{\mathds1_{I_n}}{|I_n|}\int_{I_n}\tilde\Sigma_E},
\end{equation*}
and the claim $\bb:E=-\expecm{\tilde\Sigma_E}$ easily follows by stationarity.

\medskip
\step2 Proof of strong convergences in $H^{-1}_\loc(\R^d)$.\\
For $\alpha>\frac{2d}{d+2}$, strong convergences in $H^{-1}_\loc(\R^d)$ follow from~\eqref{eq:weak-tilde-qS} and the compact Rellich embedding. It remains to consider the critical case $\alpha=\frac{2d}{d+2}$, for which we appeal to a two-scale argument inspired by~\cite[Lemma~1.15]{NSS-17}.
By stationarity, it suffices to prove $\tilde q_E(\tfrac\cdot\e)\to\Bb E$ and $\tilde\Sigma_E(\tfrac\cdot\e)\to-\bb:E$ strongly in $H^{-1}(B)$ almost surely as $\e\downarrow0$.
As the argument is the same for $\tilde q_E$ and for $\tilde\Sigma_E$, we may focus on the former.

\medskip\noindent
Let $h\in H^1(B)$ be momentarily fixed. Given $\eta>0$, we choose a partition $\{Q_i\}_i$ of~$B$ into Lipschitz subsets with $|Q_i|\simeq\eta^d$.
In these terms, we can decompose
\begin{multline}\label{eq:decomp-tildeq-B}
\int_B h\,\big(\tilde q_E(\tfrac\cdot\e)-\Bb E\big)\,=\,\sum_i\Big(\int_{Q_i} h\Big)\fint_{Q_i}\big(\tilde q_E(\tfrac\cdot\e)-\Bb E\big)\\
+\sum_i\int_{Q_i}\Big(h-\fint_{Q_i}h\Big)\,\big(\tilde q_E(\tfrac\cdot\e)-\Bb E\big).
\end{multline}
On the one hand, for all $s\in(1,\infty)$, noting that $\mathds1_{Q_i}$ belongs to $W^{\frac1s,s}(B)$, we can bound
\[\Big|\fint_{Q_i}\big(\tilde q_E(\tfrac\cdot\e)-\Bb E\big)\Big|\,\le\,\Big\|\tfrac{\mathds1_{Q_i}}{|Q_i|}\Big\|_{W^{\frac1s,s}(B)}\|\tilde q_E(\tfrac\cdot\e)-\Bb E\|_{W^{-\frac1s,s'}(B)},\]
and thus, further using the Sobolev embedding in form of $\|h\|_{\Ld^1(B)}\lesssim\|h\|_{H^1(B)}$,
we deduce for the first right-hand side term in~\eqref{eq:decomp-tildeq-B},
\begin{multline}\label{eq:1/tildeqB}
\Big|\sum_i\Big(\int_{Q_i} h\Big)\fint_{Q_i}\big(\tilde q_E(\tfrac\cdot\e)-\Bb E\big)\Big|\\
\,\lesssim\,\|h\|_{H^1(B)}\Big(\sup_i\Big\|\tfrac{\mathds1_{Q_i}}{|Q_i|}\Big\|_{W^{\frac1s,s}(B)}\Big)\|\tilde q_E(\tfrac\cdot\e)-\Bb E\|_{ W^{-\frac1s,s'}(B)}.
\end{multline}
On the other hand, using Hölder's inequality and the Poincaré--Sobolev embedding, the second right-hand side term in~\eqref{eq:decomp-tildeq-B} can be estimated as
\begin{eqnarray*}
\Big|\sum_i\int_{Q_i}\Big(h-\fint_{Q_i}h\Big)\,\big(\tilde q_E(\tfrac\cdot\e)-\Bb E\big)\Big|
&\le&\sum_i\Big\|h-\fint_{Q_i}h\Big\|_{\Ld^\frac{2d}{d-2}(Q_i)}\|\tilde q_E(\tfrac\cdot\e)-\Bb E\|_{\Ld^\frac{2d}{d+2}(Q_i)}\\
&\lesssim&\sum_i\|\nabla h\|_{\Ld^2(Q_i)}\|\tilde q_E(\tfrac\cdot\e)-\Bb E\|_{\Ld^\frac{2d}{d+2}(Q_i)}\\
&\le&\|\nabla h\|_{\Ld^2(B)}\Big(\sum_i\|\tilde q_E(\tfrac\cdot\e)-\Bb E\|_{\Ld^\frac{2d}{d+2}(Q_i)}^2\Big)^\frac12.
\end{eqnarray*}
Combining this with~\eqref{eq:decomp-tildeq-B} and~\eqref{eq:1/tildeqB}, and taking the supremum over test functions \mbox{$h\in H^1(B)$}, we conclude for all $s\in(1,\infty)$,
\begin{multline*}
\|\tilde q_E(\tfrac\cdot\e)-\Bb E\|_{H^{-1}(B)}\\
\,\lesssim\,\Big(\sup_i\Big\|\tfrac{\mathds1_{Q_i}}{|Q_i|}\Big\|_{W^{\frac1s,s}(B)}\Big)\|\tilde q_E(\tfrac\cdot\e)-\Bb E\|_{ W^{-\frac1s,s'}(B)}+\Big(\sum_i\|\tilde q_E(\tfrac\cdot\e)-\Bb E\|_{\Ld^\frac{2d}{d+2}(Q_i)}^2\Big)^\frac12.
\end{multline*}
We now pass to the limit $\e\downarrow0$ in this estimate.
Choosing $2\frac{d-1}{d-2}<s<\infty$, the compact Rellich embedding ensures that $\Ld^{\frac{2d}{d+2}}(B)$ is compactly embedded in $W^{-\frac1s,s'}(B)$. Therefore, in view of~\eqref{eq:weak-tilde-qS} with $\alpha=\frac{2d}{d+2}$, we deduce $\tilde q_E(\tfrac\cdot\e)\to\Bb E$ strongly in $W^{-\frac1s,s'}(B)$ almost surely as $\e\downarrow0$. Further using the stationarity and the boundedness of $\tilde q_E(\tfrac\cdot\e)-2\Bb E$ in $\Ld^{\frac{2d}{d+2}}(\Omega)$, cf.~Corollary~\ref{cor:extension}(i), we get almost surely
\begin{equation*}
\limsup_{\e\downarrow0}\|\tilde q_E(\tfrac\cdot\e)-\Bb E\|_{H^{-1}(B)}\,\lesssim\,\Big(\sum_i|Q_i|^\frac{d+2}d\Big)^\frac12.
\end{equation*}
Using that $\sum_i|Q_i|\lesssim1$ and $|Q_i|\lesssim\eta^d$, this turns into
\begin{equation*}
\limsup_{\e\downarrow0}\|\tilde q_E(\tfrac\cdot\e)-\Bb E\|_{H^{-1}(B)}\,\lesssim\,\eta.
\end{equation*}
Finally letting the mesh $\eta$ of the partition $\{Q_i\}_i$ tend to $0$, the conclusion follows.
\end{proof}

\subsection{Proof of Theorem~\ref{th:main-qsh}}
The moment condition~\eqref{eq:mom-2} amounts to the following: for some $2\le r\ne\frac{2d}{d-2}$ and $\frac{2d}{d+2}\le\alpha\le2\wedge\frac{2dr}{r(d-2)+2d}$, with $\alpha<\frac{d}{d-1}$ if $r=2$, the interparticle distances satisfy almost surely
\begin{equation}\label{eq:mom-2re}
\textstyle\limsup_{\e\downarrow0}\e^d\sum_{n\in\Nc_\e(U)}\mu_r(\rho_{n;U,\e})^\frac{2\alpha}{2-\alpha}\,<\,\infty.
\end{equation}
We split the proof into two steps. First, we establish the convergence of the velocity field by a direct div-curl argument inspired by the work of Jikov~\cite{Jikov-87,Jikov-90} on homogenization problems with stiff inclusions (see also~\cite[Section~3.2]{JKO94}), and then we turn to the convergence of the pressure.

\medskip
\step1 Div-curl argument: we prove that almost surely, as~$\e\downarrow0$,
\begin{equation}\label{eq:conv-hom-0}
\begin{array}{rlll}
u_\e&\cvf&\bar u,&\text{weakly in $H^1_0(U)$},\\
\tilde p_\e&\cvf&\Bb\D(\bar u),&\text{weakly in $\Ld^\alpha(U)$},\\
\tilde S_\e-\textstyle\fint_U\tilde S_\e&\cvf&\bar S,&\text{weakly in $\Ld^\alpha(U)$},
\end{array}
\end{equation}
where $(\bar u,\bar S)$ is the solution of the homogenized equation~\eqref{eq:st-hom}.
By a standard energy argument as e.g.\@ in~\cite[Step~8.1 of the proof of Proposition~2.1]{DG-19}, provided that $f\in\Ld^p(U)$ for some $p>d$, this weak convergence result easily implies the following corresponding corrector result, almost surely,
\begin{equation}\label{eq:corr-res}
\begin{array}{rlll}
\displaystyle p_\e-\sum_{E\in\Ec}q_E(\tfrac\cdot\e)\nabla_E\bar u&\to&0,&\quad\text{strongly in $\Ld^2(U)$},\\
\displaystyle u_\e-\bar u-\sum_{E\in\Ec}\e\psi_E(\tfrac\cdot\e)\nabla_E\bar u&\to&0,&\quad\text{strongly in $H^1_0(U)$},
\end{array}
\end{equation}
where we recall the short-hand notation $\nabla_E\bar u=E:\D(\bar u)$ and where $\Ec$ stands for an orthonormal basis of $\Md_0^\Sym$.
We omit the proof of this standard consequence~\eqref{eq:corr-res} and rather focus on the proof of~\eqref{eq:conv-hom-0}.

\medskip\noindent
For $\kappa>0$ we set for abbreviation $U^\kappa:=\{x\in U:\dist(x,\partial U)>\kappa\}$.
Since $q_E|_\Ic=0$ and $p_\e|_{\Ic_\e(U)}=0$, since $\tilde q_E$ and $q_E$ coincide on $\R^d\setminus\Ic$, since $\tilde p_\e$ and $p_\e$ coincide on $U\setminus\Ic_\e(U)$, and since the definition~\eqref{eq:def-Ieps} of $\Ic_\e(U)$ entails $\Ic_\e(U)\cap U^\kappa=(\e\Ic)\cap U^\kappa$ whenever $\e<\tfrac\kappa2$,
we deduce the following identity on $U^\kappa$ for $\e<\tfrac\kappa2$,
\begin{equation}\label{eq:start-divcurl}
\tilde q_E(\tfrac\cdot\e):p_\e\,=\,q_E(\tfrac\cdot\e):\tilde p_\e,
\end{equation}
and we aim at passing to the limit in both sides.
Since the energy bound entails that $(u_\e)_\e$ is almost surely bounded in $H^1_0(U)$, since Corollary~\ref{cor:extension}(ii) ensures that $(\tilde p_\e,\tilde S_\e)_\e$ is almost surely bounded in $\Ld^\alpha(U)$,
further recalling~\eqref{eq:conv-cor} and Lemma~\ref{lem:Bb-ext},
we find almost surely, up to extraction of a subsequence as $\e\downarrow0$,
\begin{equation}\label{eq:conv-divcurl}
\begin{array}{rlll}
q_E(\tfrac\cdot\e)&\cvf&E,&\text{weakly in $\Ld^2(U)$},\\
\tilde q_E(\tfrac\cdot\e)&\cvf&\Bb E,&\text{weakly in $\Ld^\alpha(U)$},\\
\tilde\Sigma_E(\tfrac\cdot\e)&\cvf&-\bb:E,&\text{weakly in $\Ld^\alpha(U)$},\\
p_\e&\cvf&\D(u_0),&\text{weakly in $\Ld^2(U)$},\\
\tilde p_\e&\cvf&\tilde p_0,&\text{weakly in $\Ld^\alpha(U)$},\\
\tilde S_\e&\cvf&\tilde S_0,&\text{weakly in $\Ld^\alpha(U)$},
\end{array}
\end{equation}
for some $u_0\in H^1_0(U)^d$, $\tilde p_0\in\Ld^\alpha(U)^{d\times d}_\Sym$, and $\tilde S_0\in\Ld^\alpha(U)$.
In case $\alpha>\frac{2d}{d+2}$ (hence~$\alpha'<\frac{2d}{d-2}$), further appealing to the compact Rellich embedding and to the sublinearity of $\psi_E^\kappa$, cf.~\eqref{eq:conv-cor}, we further deduce almost surely, up to extraction of a subsequence,
\begin{equation}\label{eq:conv-divcurl-Rellich}
\begin{array}{rlll}
\e\psi_E(\tfrac\cdot\e)&\to&0,&\text{strongly in $\Ld^{\alpha'}(U)$},\\
u_\e&\to&u_0,&\text{strongly in $\Ld^{\alpha'}(U)$}.
\end{array}
\end{equation}
If the inclusions $\{I_n\}_n$ were uniformly separated as assumed in~\cite{DG-19}, then we could choose \mbox{$\alpha=2$}, cf.~\eqref{eq:mom-2re}, so that a standard div-curl argument in form of e.g.~\cite[Lemma~12.12]{JKO94} would allow to use~\eqref{eq:conv-divcurl} and pass to the limit in both sides of identity~\eqref{eq:start-divcurl} (along the subsequence), to the effect of
\begin{equation}\label{eq:start-divcurl-lim}
\Bb E:\D(u_0)\,=\,E:\tilde p_0,\qquad\text{in $U$}.
\end{equation}
In the present situation, with $\alpha<2$, we need to repeat the proof of the div-curl lemma and show that this identity~\eqref{eq:start-divcurl-lim} still holds.
Once this is proven, the conclusion~\eqref{eq:conv-hom-0} easily follows: passing to the weak limit in~\eqref{eq:pS-ext-eqn} (along the subsequence) yields
\[-\Div(2\tilde p_0-\tilde S_0\Id)=(1-\lambda)f,\qquad\text{in $U$},\]
and thus, inserting~\eqref{eq:start-divcurl-lim} in form of $\tilde p_0=\Bb\D(u_0)$, we deduce that $(u_0,\tilde S_0-\fint_U\tilde S_0)$ coincides with the unique solution $(\bar u,\bar S)$ of the homogenized equation~\eqref{eq:st-hom}. With this characterization of the limit, the conclusion~\eqref{eq:conv-hom-0} now follows from~\eqref{eq:conv-divcurl}.

\medskip\noindent
It remains to prove~\eqref{eq:start-divcurl-lim}, and we split the proof in two further substeps. We start with the case $\frac{2d}{d+2}<\alpha<2$, and next we discuss the critical case~$\alpha=\frac{2d}{d+2}$.

\medskip
\substep{1.1} Proof of~\eqref{eq:start-divcurl-lim} in case $\frac{2d}{d+2}<\alpha<2$.\\
We shall pass to the limit in both sides of~\eqref{eq:start-divcurl} and we start with the analysis of the left-hand side. Given a test function $h\in C^1_c(U)$ supported in $U^\kappa$ for some fixed $\kappa>2\e$, integrating by parts, using the property~\eqref{eq:psiSig-ext-eqn} of the extension $(\tilde q_E,\tilde\Sigma_E)$, and using the constraint $\Div(u_\e)=0$, we find
\begin{eqnarray}
\int_U h\,\tilde q_E(\tfrac\cdot\e):p_\e&=&\int_Uh\,\tilde q_E(\tfrac\cdot\e):\D(u_\e)\nonumber\\
&=&\int_U\D(hu_\e):\tilde q_E(\tfrac\cdot\e)-\int_U (\nabla h\otimes u_\e):\tilde q_E(\tfrac\cdot\e)\nonumber\\
&=&\frac12\int_U\tilde \Sigma_E(\tfrac\cdot\e)\,\Div(hu_\e)-\int_U (\nabla h\otimes u_\e):\tilde q_E(\tfrac\cdot\e)\nonumber\\
&=&-\frac12\int_U (\nabla h\otimes u_\e):\big(2\tilde q_E-\tilde\Sigma_E\Id\big)(\tfrac\cdot\e).\label{eq:rewr-LHS-qp}
\end{eqnarray}
Note that the relation $\Div(u_\e)=0$ entails $\int_U\nabla h\cdot u_\e=0$, which allows to add any constant to the pressure $\tilde\Sigma_E$, for instance replacing it by $\tilde\Sigma_E-\expecm{\tilde\Sigma_E}$.
In view of~\eqref{eq:conv-divcurl} and~\eqref{eq:conv-divcurl-Rellich}, we may now pass to the limit in the above, to the effect of
\begin{equation}\label{eq:lhs-lim}
\lim_{\e\downarrow0}\int_U h\,\tilde q_E(\tfrac\cdot\e):p_\e
\,=\,-\int_U (\nabla h\otimes u_0):\Bb E
\,=\,\int_U h\,\Bb E:\D(u_0).
\end{equation}
We turn to the analysis of the right-hand side of~\eqref{eq:start-divcurl}.
Integrating by parts, using the property~\eqref{eq:pS-ext-eqn} of the extension $(\tilde p_\e,\tilde S_\e)$, and using the constraint $\Div(\psi_E)=0$, we find
\begin{eqnarray}
\int_U h\, q_E(\tfrac\cdot\e):\tilde p_\e&=&E:\int_Uh\,\tilde p_\e+\int_Uh\, \D(\psi_E)(\tfrac\cdot\e):\tilde p_\e\nonumber\\
&=&E:\int_Uh\,\tilde p_\e+\int_U \D\big(h\,\e\psi_E(\tfrac\cdot\e)\big):\tilde p_\e-\int_U \big(\nabla h\otimes\e\psi_E(\tfrac\cdot\e)\big):\tilde p_\e\nonumber\\
&=&E:\int_Uh\,\tilde p_\e+\frac12\int_{U\setminus\Ic_\e(U)}h\,\e\psi_E(\tfrac\cdot\e)\cdot f\nonumber\\
&&\qquad-\frac12\int_U \big(\nabla h\otimes\e\psi_E(\tfrac\cdot\e)\big):\big(2\tilde p_\e-\tilde S_\e\Id\big).\label{eq:rewr-RHS-qp}
\end{eqnarray}
In view of~\eqref{eq:conv-divcurl} and~\eqref{eq:conv-divcurl-Rellich}, we may now pass to the limit in the above, to the effect of
\begin{equation}\label{eq:rhs-lim}
\lim_{\e\downarrow0}\int_U h\, q_E(\tfrac\cdot\e):\tilde p_\e\,=\,E:\int_Uh\,\tilde p_0.
\end{equation}
Combining this with~\eqref{eq:start-divcurl} and~\eqref{eq:lhs-lim}, and choosing an arbitrary test function $h\in C^\infty_c(U)$, this proves the claim~\eqref{eq:start-divcurl-lim}.

\medskip
\substep{1.2} Proof of~\eqref{eq:start-divcurl-lim} in the critical case $\alpha=\frac{2d}{d+2}$.\\
It suffices to prove that~\eqref{eq:lhs-lim} and~\eqref{eq:rhs-lim} still hold in this case.
Due to the failure of the compact Rellich embedding~\eqref{eq:conv-divcurl-Rellich}, we can no longer pass to the limit directly in~\eqref{eq:rewr-LHS-qp} and~\eqref{eq:rewr-RHS-qp}, so a finer analysis is needed. We appeal again to a two-scale argument as inspired by~\cite[Lemma~1.15]{NSS-17}.

\medskip\noindent
We start with the proof of~\eqref{eq:lhs-lim}.
Given $\eta>0$, we choose a partition $\{Q_i\}_i$ of $U$ into measurable subsets with $|Q_i|\simeq\eta^d$. In these terms, we can decompose~\eqref{eq:rewr-LHS-qp} as
\begin{multline}\label{eq:decomp-qp-Rell}
\int_U h\,\tilde q_E(\tfrac\cdot\e):p_\e\,=\,-\frac12\sum_i\Big(\fint_{Q_i}u_\e\Big)\cdot\int_{Q_i}\big(2\tilde q_E-\tilde\Sigma_E\Id\big)(\tfrac\cdot\e)\,\nabla h\\
-\frac12\sum_i\int_{Q_i} \nabla h\otimes\Big( u_\e-\fint_{Q_i}u_\e\Big):\big(2\tilde q_E-\tilde\Sigma_E\Id\big)(\tfrac\cdot\e).
\end{multline}
On the one hand, using the compact Rellich embedding in form of the almost sure strong convergence $u_\e\to u_0$ in~$\Ld^1(U)$, and using Lemma~\ref{lem:Bb-ext}, we find
\begin{multline*}
\lim_{\e\downarrow0}\frac12\sum_i\Big(\fint_{Q_i}u_\e\Big)\cdot\int_{Q_i} \nabla h\cdot\big(2\tilde q_E-\tilde\Sigma_E\Id\big)(\tfrac\cdot\e)\\
\,=\,\frac12\sum_i\Big(\int_{Q_i} \nabla h\Big)\otimes\Big(\fint_{Q_i}u_0\Big):\big(2\Bb E+(\bb:E) \Id\big),
\end{multline*}
hence, letting the mesh $\eta$ of the partition $\{Q_i\}_i$ tend to $0$, using that the constraint \mbox{$\Div(u_\e)=0$} entails $\int_U\nabla h\cdot u_0=0$, and integrating by parts,
\begin{multline}\label{eq:qp-part1-spl-lim}
\lim_{\eta\downarrow0}\lim_{\e\downarrow0}\frac12\sum_i\Big(\fint_{Q_i}u_\e\Big)\cdot\int_{Q_i} \nabla h\cdot\big(2\tilde q_E-\tilde\Sigma_E\Id\big)(\tfrac\cdot\e)\\
\,=\,\frac12\Big(\int_{U} \nabla h\otimes u_0\Big):\big(2\Bb E+(\bb:E) \Id\big)
\,=\,-\int_U h\,\Bb E:\D(u_0).
\end{multline}
On the other hand, using Hölder's inequality and the Poincaré--Sobolev embedding, the second right-hand side term in~\eqref{eq:decomp-qp-Rell} can be estimated as
\begin{eqnarray*}
\lefteqn{\Big|\sum_i\int_{Q_i} \nabla h\otimes\Big( u_\e-\fint_{Q_i}u_\e\Big):\big(2\tilde q_E-\tilde\Sigma_E\Id\big)(\tfrac\cdot\e)\Big|}\\
&\le&\|\nabla h\|_{\Ld^\infty(U)}\sum_i\Big\|u_\e-\fint_{Q_i}u_\e\Big\|_{\Ld^\frac{2d}{d-2}(Q_i)}\|(\tilde q_E,\tilde\Sigma_E)(\tfrac\cdot\e)\|_{\Ld^\frac{2d}{d+2}(Q_i)}\\
&\lesssim&\|\nabla h\|_{\Ld^\infty(U)}\sum_i\|\nabla u_\e\|_{\Ld^2(Q_i)}\|(\tilde q_E,\tilde\Sigma_E)(\tfrac\cdot\e)\|_{\Ld^\frac{2d}{d+2}(Q_i)}\\
&\le&\|\nabla h\|_{\Ld^\infty(U)}\|\nabla u_\e\|_{\Ld^2(U)}\Big(\sum_i\|(\tilde q_E,\tilde\Sigma_E)(\tfrac\cdot\e)\|_{\Ld^\frac{2d}{d+2}(Q_i)}^2\Big)^\frac12,
\end{eqnarray*}
hence, passing to the limit $\e\downarrow0$, using the boundedness of $\nabla u_\e$ in $\Ld^2(U)$, and using the stationarity and the boundedness of $(\tilde q_E,\tilde \Sigma_E)$ in $\Ld^\frac{2d}{d+2}(\Omega)$, cf.~Corollary~\ref{cor:extension}(i),
\begin{multline*}
\limsup_{\e\downarrow0}\Big|\sum_i\int_{Q_i} \nabla h\otimes\Big( u_\e-\fint_{Q_i}u_\e\Big):\big(2\tilde q_E-\tilde\Sigma_E\Id\big)(\tfrac\cdot\e)\Big|\\
\,\lesssim_f\,\|\nabla h\|_{\Ld^\infty(U)}\Big(\sum_i|Q_i|^\frac{d+2}d\Big)^\frac12\,\lesssim\,\eta\|\nabla h\|_{\Ld^\infty(U)}.
\end{multline*}
Now letting the mesh $\eta$ of the partition $\{Q_i\}_i$ tend to $0$, and combining this with~\eqref{eq:decomp-qp-Rell} and~\eqref{eq:qp-part1-spl-lim}, we deduce~\eqref{eq:lhs-lim}.

\medskip\noindent
We turn to the proof of~\eqref{eq:rhs-lim}.
Given $\eta>0$, we consider as above a partition $\{Q_i\}_i$ of $U$ into measurable subsets with $|Q_i|\simeq\eta^d$.
Starting point is the Poincaré--Sobolev embedding in the form
\[\|\e\psi_{E}(\tfrac\cdot\e)\|_{\Ld^\frac{2d}{d-2}(Q_i)}\,\lesssim\,\|\nabla\psi_E(\tfrac\cdot\e)\|_{\Ld^2(Q_i)}+|Q_i|^{\frac{d-2}{2d}}\fint_{Q_i}|\e\psi_E(\tfrac\cdot\e)|.\]
By the stationarity and the boundedness of $\nabla\psi_E$ in $\Ld^2(\Omega)$, and by the sublinearity of $\psi_{E}$ in~$\Ld^1$, cf.~\eqref{eq:conv-cor}, we deduce almost surely
\[\limsup_{\e\downarrow0}\|\e\psi_{E}(\tfrac\cdot\e)\|_{\Ld^\frac{2d}{d-2}(Q_i)}\,\lesssim\,|Q_i|^{\frac12}\|\nabla\psi_E\|_{\Ld^2(\Omega)}\,\lesssim\,|Q_i|^\frac12.\]
Summing over $i$, this yields
\begin{equation*}
\limsup_{\e\downarrow0}\|\e\psi_{E}(\tfrac\cdot\e)\|_{\Ld^\frac{2d}{d-2}(U)}
\,\lesssim\,\Big(\sum_i|Q_i|^\frac{d}{d-2}\Big)^\frac{d-2}{2d}
\,\lesssim\,\eta,
\end{equation*}
and thus, letting the mesh $\eta$ of the partition $\{Q_i\}_i$ tend to $0$,
\begin{equation}
\lim_{\e\downarrow0}~\|\e\psi_{E}(\tfrac\cdot\e)\|_{\Ld^\frac{2d}{d-2}(U)}\,=\,0,
\end{equation}
which proves that $\psi_{E}$ is in fact still sublinear in $\Ld^{\alpha'}=\Ld^\frac{2d}{d-2}$.
This allows to pass to the limit in~\eqref{eq:rewr-RHS-qp}, and the claim~\eqref{eq:rhs-lim} follows.

\medskip
\step2 Convergence of the pressure.\\
While it is already shown in Step~1, cf.~\eqref{eq:conv-hom-0}, that almost surely $\tilde S_\e-\fint_U\tilde S_\e\cvf{}\bar S$ weakly in~$\Ld^\alpha(U)$, we turn to the weak convergence of the restricted pressure $S_\e\mathds1_{U\setminus\Ic_\e(U)}=\tilde S_\e\mathds1_{U\setminus\Ic_\e(U)}$,
and we establish at the same time the corrector result for the pressure, cf.~\eqref{eq:conv2sc}. For that purpose, we start by examining the two-scale expansion errors
\begin{eqnarray*}
w_\e&:=&u_\e-\bar u-\sum_{E\in\Ec}\e\psi_E(\tfrac\cdot\e)\nabla_E\bar u,\\
Q_\e&:=&S_\e\mathds1_{U\setminus\Ic_\e(U)}-\bar S-\bb:\D(\bar u)-\sum_{E\in\Ec}(\Sigma_E\mathds1_{\R^d\setminus\Ic})(\tfrac\cdot\e)\nabla_E\bar u.
\end{eqnarray*}
Without loss of generality, we may assume that $f\in W^{1,\infty}(U)^d$ and $\bar u\in W_0^{3,\infty}(U)^d$, while the general case easily follows by an approximation argument as in~\cite[Step~8.4 of the proof of Proposition~2.1]{DG-19}.

\medskip\noindent
Consider a test function $g\in C^\infty_c(U)^d$ with $\D(g)|_{\e\Ic}=0$. Inserting the above definition of $(w_\e,Q_\e)$ and reorganizing the terms, we compute
\begin{multline*}
\int_U\D(g):\big(2\D(w_\e)-Q_\e\Id\big)
\,=\,\int_U\D(g):\big(2p_\e-S_\e\Id\big)-\int_U\D(g):\big(2\Bb\D(\bar u)-\bar S\Id\big)\\
-\sum_{E\in\Ec}\int_U\D(g):\big(2q_E-\Sigma_E\mathds1_{\R^d\setminus\Ic}\Id\big)(\tfrac\cdot\e)\nabla_E\bar u
+\sum_{E\in\Ec}\int_U\D(g):\big(2\Bb E+(\bb:E)\Id\big)\nabla_E\bar u\\
-2\sum_{E\in\Ec}\int_U\D(g):\big(\nabla\nabla_E\bar u\otimes\e\psi_E(\tfrac\cdot\e)\big).
\end{multline*}
Since $\D(g)$ vanishes in $\e\Ic$, recalling that $(q_E,\Sigma_E)(\tfrac\cdot\e)$ and $(p_\e,S_\e)$ coincide with $(\tilde q_E,\tilde \Sigma_E)(\tfrac\cdot\e)$ and $(\tilde p_\e,\tilde S_\e)$ in $U\setminus\e\Ic\subset U\setminus\Ic_\e(U)$, and appealing to~\eqref{eq:psiSig-ext-eqn} and~\eqref{eq:pS-ext-eqn}, and to the homogenized equation~\eqref{eq:st-hom}, we easily find
\begin{equation}\label{eq:ident-calF}
\int_U\D(g):\big(2\D(w_\e)-Q_\e\Id\big)\,=\,\calF_\e(g),
\end{equation}
in terms of
\begin{multline}\label{eq:def-calF}
\calF_\e(g)\,:=\,-\int_Ug\cdot\big(\mathds1_{\Ic}(\tfrac\cdot\e)-\lambda\big) f-2\sum_{E\in\Ec}\int_U\D(g):\big(\nabla\nabla_E\bar u\otimes\e\psi_E(\tfrac\cdot\e)\big)\\
+\sum_{E\in\Ec}\int_U\big(\nabla\nabla_E\bar u\otimes g\big):\Big(\big(2\tilde q_E-\tilde\Sigma_E\Id\big)(\tfrac\cdot\e)
-\big(2\Bb E+(\bb:E)\Id\big)\Big).
\end{multline}
We now appeal to Lemma~\ref{lem:Bog2} in the following form: there exists $z_\e\in H^1_0(U)^d$ with $\D(z_\e)|_{\e\Ic}=0$, such that
\[\Div(z_\e)=\Big(T_\e|T_\e|^{\alpha-2}-\textstyle\fint_{U\setminus\e\Ic}T_\e|T_\e|^{\alpha-2}\Big)\mathds1_{U\setminus\e\Ic},\qquad T_\e:=Q_\e-\fint_{U\setminus\e\Ic}Q_\e,\]
and
\begin{eqnarray}
\|\nabla z_\e\|_{\Ld^2(U)}&\lesssim_{U,\alpha,r}&\Lambda_\e(U;r,\tfrac{2\alpha}{2-\alpha})\,\|T_\e|T_\e|^{\alpha-2}\|_{\Ld^{\alpha'}(U\setminus\e\Ic)}\nonumber\\
&\lesssim&\Lambda_\e(U;r,\tfrac{2\alpha}{2-\alpha})\,\|Q_\e-\textstyle\fint_{U\setminus\e\Ic}Q_\e\|_{\Ld^{\alpha}(U\setminus\e\Ic)}^{\alpha-1},\label{eq:bnd-zeps-bab}
\end{eqnarray}
where we have set
\[\Lambda_\e(U;r,p)\,:=\,\Big(|U|+\e^d\sum_{n:\e I_n\cap U\ne\varnothing}\mu_{r}(\rho_{n;U,\e})^{p}\Big)^{\frac1p}.\]
Testing~\eqref{eq:ident-calF} with $g=z_\e$, and using the properties of $z_\e$, we find
\begin{equation}\label{eq:pretest-Qepsz}
\|Q_\e-{\textstyle\fint_{U\setminus\e\Ic}Q_\e}\|_{\Ld^\alpha(U\setminus\e\Ic)}^\alpha\,=\,-\calF_\e(z_{\e})+2\int_{U}\D(z_{\e}):\D(w_\e).
\end{equation}
Noting that the definition~\eqref{eq:def-calF} of $\calF_\e$ yields
\begin{multline*}
|\calF_\e(g)|\,\lesssim\,\|g\|_{H^1(U)}\Big(\|f\|_{W^{1,\infty}(U)}+\|\nabla\bar u\|_{W^{2,\infty}(U)}\Big)\\
\times\sup_{E\in\Ec}\Big(\|\e\psi_E(\tfrac\cdot\e)\|_{\Ld^2(U)}+\|\mathds1_{\Ic}(\tfrac\cdot\e)-\lambda\|_{H^{-1}(U)}\\
+\|\tilde q_E(\tfrac\cdot\e)-\Bb E\|_{H^{-1}(U)}
+\|\tilde\Sigma_E(\tfrac\cdot\e)+\bb:E\|_{H^{-1}(U)}\Big),
\end{multline*}
inserting this into~\eqref{eq:pretest-Qepsz},
and using~\eqref{eq:bnd-zeps-bab}, we deduce
\begin{multline*}
\|Q_\e-{\textstyle\fint_{U\setminus\e\Ic}Q_\e}\|_{\Ld^\alpha(U\setminus\e\Ic)}\,\lesssim_{U,\alpha,r}\,\Lambda_\e(U;r,\tfrac{2\alpha}{2-\alpha})\,\|w_\e\|_{H^1(U)}\\
+\Lambda_\e(U;r,\tfrac{2\alpha}{2-\alpha})\,\Big(\|f\|_{W^{1,\infty}(U)}+\|\nabla\bar u\|_{W^{2,\infty}(U)}\Big)\\
\times\sup_{E\in\Ec}\Big(\|\e\psi_E(\tfrac\cdot\e)\|_{\Ld^2(U)}+\|\mathds1_{\Ic}(\tfrac\cdot\e)-\lambda\|_{H^{-1}(U)}\\
+\|\tilde q_E(\tfrac\cdot\e)-\Bb E\|_{H^{-1}(U)}
+\|\tilde\Sigma_E(\tfrac\cdot\e)+\bb:E\|_{H^{-1}(U)}\Big).
\end{multline*}
Noting that the moment condition~\eqref{eq:mom-2re} entails $\limsup_{\e\downarrow0}\Lambda_\e(U;r,\frac{2\alpha}{2-\alpha})<\infty$, and using~\eqref{eq:conv-cor}, \eqref{eq:corr-res}, and Lemma~\ref{lem:Bb-ext}, together with the ergodic theorem in form of the almost sure weak convergence $\mathds1_\Ic(\tfrac\cdot\e)\cvf\lambda$ in $\Ld^2_\loc(\R^d)$,
the above right-hand side tends to $0$ almost surely as $\e\downarrow0$.
This concludes the proof of~\eqref{eq:conv2sc}.
\qed

\section{Further technical tools}

This last section is devoted to the proof of Corollaries~\ref{lem:trace} and~\ref{lem:cacc}, which are further technical tools for the analysis of particle suspensions without uniform separation.

\begin{proof}[Proof of Corollary~\ref{lem:trace}]
Note that the Stokes equation~\eqref{eq:uS} entails $\Div(\sigma(u,S))=0$ in $\R^d\setminus\Ic$.
For all $n$, in terms of the cut-off function $w_n\in H^1_0(I_n^+)$ with \mbox{$w_n|_{I_n}=1$} that we have constructed in Lemma~\ref{lem:wn}, an integration by parts then yields
\begin{equation}\label{eq:befexttrace}
\int_{\partial I_n}g\cdot\sigma(u,S)\nu\,=\,-\int_{I_n^+\setminus I_n}\Div(w_n\sigma(u,S)g)
\,=\,-\int_{I_n^+\setminus I_n}\D(w_n g):\sigma(u,S).
\end{equation}
In order to reformulate the right-hand side, we appeal to the extension result of Theorem~\ref{th:extension}.
More precisely, given $\beta\in(1,\infty)$ and $\alpha,r$ as in~\eqref{eq:ch-param}, since the Stokes equation~\eqref{eq:uS} ensures that the flux $p=\D(u)\in\Ld^2_\loc(\R^d)^{d\times d}_\Sym$ satisfies $\Tr(p)=0$ and
\[\int_{\R^d}\D(g):p=0,\qquad\forall g\in C^1_c(\R^d)^d:\Div(g)=0,\,\D(g)|_\Ic=0,\]
Theorem~\ref{th:extension} provides an extension $\tilde p\in\Ld^\alpha(I_n^+)^{d\times d}_\Sym$ with $\Tr(\tilde p)=0$, and an associated pressure field $\tilde S\in\Ld^\alpha_\loc(\R^d)$, such that
\[(\tilde p,\tilde S)|_{\R^d\setminus\Ic}=(p,S)|_{\R^d\setminus\Ic},\qquad\text{and}\qquad\Div(2\tilde p-\tilde S\Id)=0,\quad\text{in $\R^d$},\]
and such that the following estimate holds, for all $n$,
\begin{equation*}
\|(\tilde p,\tilde S)\|_{\Ld^\alpha(I_n^+)}\,\lesssim_{\alpha,\beta,r}\,\mu_{r}(\rho_n)\,\|\!\D(u)\|_{\Ld^\beta(I_n^+\setminus I_n)}.
\end{equation*}
Writing $\sigma(u,S)=2p-S\Id$ in~\eqref{eq:befexttrace}, and using these extensions, we find
\begin{equation*}
\int_{\partial I_n}g\cdot\sigma(u,S)\nu
\,=\,-\int_{I_n^+\setminus I_n}\D(w_n g):(2p- S\Id)
\,=\,\int_{I_n}\D(w_n g):(2\tilde p-\tilde S\Id),
\end{equation*}
and we may then estimate
\begin{eqnarray*}
\Big|\int_{\partial I_n}g\cdot\sigma(u,S)\nu\Big|
&\lesssim&\|w_n\|_{W^{1,\alpha'}(I_n^+)}\|g\|_{W^{1,\infty}(I_n^+)}\|(\tilde p,\tilde S)\|_{\Ld^\alpha(I_n)}\\
&\lesssim_{\alpha,\beta,r}&\mu_{r}(\rho_n)\|w_n\|_{W^{1,\alpha'}(I_n^+)}\|g\|_{W^{1,\infty}(I_n^+)}\|\!\D(u)\|_{\Ld^\beta(I_n^+\setminus I_n)}.
\end{eqnarray*}
Combining this with the bound on norms of $w_n$ in Lemma~\ref{lem:wn}, choosing $\beta=2$, and optimizing the choice of $\alpha,r$, the conclusion follows.
\end{proof}

\begin{proof}[Proof of Corollary~\ref{lem:cacc}]
For $R\ge5$,
choose $\zeta_R\in C^\infty_c(B_{2R-4};\R^+)$ with $\zeta_R|_{B_R}=1$ and with $|\nabla\zeta_R|\lesssim R^{-1}$.
For any $V\in\R^d$ and $c\in\R$, testing the Stokes equation~\eqref{eq:uS} with $\zeta_R(u-V)$, and replacing the pressure $S$ by $S-c$, we find
\begin{multline*}
\int_{\R^d}\zeta_R|\nabla u|^2\,=\,
-\int_{\R^d}\big((u-V)\otimes\nabla\zeta_R\big):\big(\nabla u-(S-c)\Id\mathds1_{\R^d\setminus\Ic}\big)\\
-\sum_{n:I_n^+\subset B_{2R}}\int_{\partial I_n}\zeta_R\,(u-V)\cdot\sigma(u,S-c)\nu.
\end{multline*}
Since $\D(u)=0$ in $I_n$, we may write $u=V_n+\Theta_n(x-x_n)$ in $I_n$ for some $V_n\in\R^d$ and $\Theta_n\in\Md^\Skew$.
The boundary conditions for $u$ then allow to add any constant to the test function $\zeta_R$ in the last right-hand side term, and we obtain
\begin{multline*}
\int_{\R^d}\zeta_R|\nabla u|^2\,=\,
-\int_{\R^d}\big((u-V)\otimes\nabla\zeta_R\big):\big(\nabla u-(S-c)\Id\mathds1_{\R^d\setminus\Ic}\big)\\
-\sum_{n:I_n^+\subset B_{2R}}\int_{\partial I_n}\Big(\zeta_R-\fint_{I_n}\zeta_R\Big)\big(V_n-V+\Theta_n(x-x_n)\big)\cdot\sigma(u,S-c)\nu.
\end{multline*}
Hence, using the properties of $\zeta_R$, Hölder's inequality, and appealing to the trace estimate of Corollary~\ref{lem:trace} to bound the last right-hand side term,
we deduce for all $s\ge1$,
\begin{multline*}
\|\nabla u\|_{\Ld^2(B_R)}^2
\,\lesssim\,
R^{-1}\|u-V\|_{\Ld^{s}(B_{2R})}\Big(\|\nabla u\|_{\Ld^{s'}(B_{2R})}+\|S-c\|_{\Ld^{s'}(B_{2R}\setminus\Ic)}\Big)\\
+R^{-1}\Big(\sum_{n:I_n^+\subset B_{2R}}\mu'(\rho_n)^2(|V_n-V|^2+|\Theta_n|^2)\Big)^\frac12\|\nabla u\|_{\Ld^2(B_{2R})},
\end{multline*}
where we have set for abbreviation,
\[\mu'(\rho_n):=\left\{\begin{array}{lll}
\rho_n^{\frac{1}{4d}(d+1)(d+2)-\frac52}&:&d\le6,\\
1&:&d>6.
\end{array}\right.\]
Choosing $c=\fint_{B_{2R}\setminus\Ic}S$ and appealing to a pressure estimate as in~\eqref{eq:pres-estimate} (with $\alpha=s'$), this becomes for all $2\le r\ne\frac{2d}{d-2}$ and $2\vee\frac{2dr}{d(r-2)+2r}\le s<\infty$, with $s>d$ if $r=2$,
\begin{multline*}
\|\nabla u\|_{\Ld^2(B_R)}^2
\,\lesssim\,
R^{-1}\Big(|B_R|+\sum_{n:I_n^+\subset B_{2R}}\mu_r(\rho_n)^\frac{2s}{s-2}\Big)^\frac{s-2}{2s}\|u-V\|_{\Ld^{s}(B_{2R})}\|\nabla u\|_{\Ld^2(B_{2R})}\\
+R^{-1}\Big(\sum_{n:I_n^+\subset B_{2R}}\mu'(\rho_n)^2(|V_n-V|^2+|\Theta_n|^2)\Big)^\frac12\|\nabla u\|_{\Ld^2(B_{2R})}.
\end{multline*}
Noting that
\[|V_n-V|^2\lesssim\int_{I_n}|u-V|^2,\qquad|\Theta_n|^2\lesssim\int_{I_n}|\nabla u|^2,\]
Hölder's inequality yields
\begin{eqnarray*}
\lefteqn{\sum_{n:I_n^+\subset B_{2R}}\mu'(\rho_n)^2(|V_n-V|^2+|\Theta_n|^2)}\\
&\lesssim&\Big(\sum_{n:I_n^+\subset B_{2R}}\mu'(\rho_n)^\frac{2s}{s-2}\Big)^\frac{s-2}s\|u-V\|_{\Ld^s(B_{2R})}^2+\Big(\sup_{n:I_n^+\subset B_{2R}}\mu'(\rho_n)^2\Big)\|\nabla u\|_{\Ld^2(B_{2R})}^2\\
&\lesssim&\Big(\sum_{n:I_n^+\subset B_{2R}}\mu'(\rho_n)^\frac{2s}{s-2}\Big)^\frac{s-2}s\Big(\|u-V\|_{\Ld^s(B_{2R})}^2+\|\nabla u\|_{\Ld^2(B_{2R})}^2\Big).
\end{eqnarray*}
Inserting this into the above, choosing $V:=\fint_{B_{2R}}u$, and optimizing in $r$, the conclusion follows.
\end{proof}

\section*{Acknowledgements}
We thank David Gérard-Varet for pointing out a mistake in a previous version of this work and for explaining some related computations in~\cite{GVH-12}.
We also thank Antoine Gloria for motivating discussions on the topic, we thank Roxane Verdikt for drawing figures, and we acknowledge financial support from the CNRS-Momentum program.

\bibliographystyle{plain}

\begin{thebibliography}{10}

\bibitem{BG-72}
G.~K. Batchelor and J.T. Green.
\newblock The determination of the bulk stress in suspension of spherical
  particles to order $c^2$.
\newblock {\em J. Fluid Mech.}, 56(3):401--427, 1972.

\bibitem{BG-72a}
G.~K. Batchelor and J.T. Green.
\newblock The hydrodynamic interaction of two small freely-moving spheres in a
  linear flow field.
\newblock {\em J. Fluid Mech.}, 56(2):375--400, 1972.

\bibitem{BFO-18}
P.~Bella, B.~Fehrman, and F.~Otto.
\newblock A {L}iouville theorem for elliptic systems with degenerate ergodic
  coefficients.
\newblock {\em Ann. Appl. Probab.}, 28(3):1379--1422, 2018.

\bibitem{Bella-Schaffner-19}
P.~Bella and M.~Sch\"affner.
\newblock Local boundedness and {H}arnack inequality for solutions of linear
  non-uniformly elliptic equations.
\newblock {\em Comm. Pure Appl. Math.}, 74(3):453--477, 2021.

\bibitem{Chiarini-Deuschel-16}
A.~Chiarini and J.-D. Deuschel.
\newblock Invariance principle for symmetric diffusions in a degenerate and
  unbounded stationary and ergodic random medium.
\newblock {\em Ann. Inst. Henri Poincar\'{e} Probab. Stat.}, 52(4):1535--1563,
  2016.

\bibitem{DG21}
M.~Duerinckx and A.~Gloria.
\newblock Continuum percolation in stochastic homogenization and the effective
  viscosity problem.
\newblock Preprint, arXiv:2108.09654.

\bibitem{DG-20c}
M.~Duerinckx and A.~Gloria.
\newblock On {E}instein's effective viscosity formula.
\newblock Preprint, arXiv:2008.03837.

\bibitem{DG-20b}
M.~Duerinckx and A.~Gloria.
\newblock Quantitative homogenization theory for random suspensions in steady
  {S}tokes flow.
\newblock Preprint, arXiv:2103.06414.

\bibitem{DG-20a}
M.~Duerinckx and A.~Gloria.
\newblock Sedimentation of random suspensions and the effect of
  hyperuniformity.
\newblock Preprint, arXiv:2004.03240.

\bibitem{DG-19}
M.~Duerinckx and A.~Gloria.
\newblock Corrector equations in fluid mechanics: {E}ffective viscosity of
  colloidal suspensions.
\newblock {\em Arch. Ration. Mech. Anal.}, 2020.

\bibitem{FHS-19}
F.~Flegel, M.~Heida, and M.~Slowik.
\newblock Homogenization theory for the random conductance model with
  degenerate ergodic weights and unbounded-range jumps.
\newblock {\em Ann. Inst. Henri Poincar\'{e} Probab. Stat.}, 55(3):1226--1257,
  2019.

\bibitem{Francfort-92}
G.~A. Francfort.
\newblock Homogenisation of a class of fourth order equations with application
  to incompressible elasticity.
\newblock {\em Proc. Roy. Soc. Edinburgh}, 120A:25--46, 1992.

\bibitem{Galdi}
G.~P. Galdi.
\newblock {\em An introduction to the mathematical theory of the
  {N}avier-{S}tokes equations. Steady-state problems}.
\newblock Springer Monographs in Mathematics. Springer, New York, second
  edition, 2011.

\bibitem{GV-20}
D.~G\'erard-Varet.
\newblock {Derivation of Batchelor--Green formula for random suspensions}.
\newblock Preprint, arXiv:2008.06324.

\bibitem{GVH}
D.~{G\'erard-Varet} and M.~{Hillairet}.
\newblock {Analysis of the viscosity of dilute suspensions beyond {Einstein's
  formula}}.
\newblock Preprint, arXiv:1905.08208.

\bibitem{GVH-12}
D.~G\'{e}rard-Varet and M.~Hillairet.
\newblock Computation of the drag force on a sphere close to a wall: the
  roughness issue.
\newblock {\em ESAIM Math. Model. Numer. Anal.}, 46(5):1201--1224, 2012.

\bibitem{GV-Hofer-20}
D.~G\'erard-Varet and R.~M. H\"ofer.
\newblock Mild assumptions for the derivation of {E}instein's effective
  viscosity formula.
\newblock Preprint, arXiv:2002.04846.

\bibitem{GVM-20}
D.~G\'erard-Varet and A.~Mecherbet.
\newblock {On the correction to Einstein's formula for the effective
  viscosity}.
\newblock arXiv:2004.05601.

\bibitem{Girodroux-GV-21}
D.~Gérard-Varet and A.~Girodroux-Lavigne.
\newblock Homogenization of stiff inclusions through network approximation.
\newblock Preprint, arXiv:2106.06299.

\bibitem{Jikov-87}
V.~V. Jikov.
\newblock {Averaging of functionals in the calculus of variations and
  elasticity}.
\newblock {\em Math. USSR, Izvestiya}, 29:33--66, 1987.

\bibitem{Jikov-90}
V.~V. Jikov.
\newblock {Some problems of extension of functions arising in connection with
  the homogenization theory}.
\newblock {\em Diff. Uravnenia}, 26(1):39--51, 1990.

\bibitem{JKO94}
V.~V. Jikov, S.~M. Kozlov, and O.~A. Ole{\u\i}nik.
\newblock {\em Homogenization of differential operators and integral
  functionals}.
\newblock Springer-Verlag, Berlin, 1994.

\bibitem{NSS-17}
S.~Neukamm, M.~Sch\"{a}ffner, and A.~Schl\"{o}merkemper.
\newblock Stochastic homogenization of nonconvex discrete energies with
  degenerate growth.
\newblock {\em SIAM J. Math. Anal.}, 49(3):1761--1809, 2017.

\bibitem{Niethammer-Schubert-19}
B.~{Niethammer} and R.~{Schubert}.
\newblock {A local version of Einstein's formula for the effective viscosity of
  suspensions}.
\newblock Preprint, arXiv:1903.08554.

\bibitem{Opic-Kufner-90}
B.~Opic and A.~Kufner.
\newblock {\em Hardy-type inequalities}, volume 219 of {\em Pitman Research
  Notes in Mathematics Series}.
\newblock Longman Scientific \& Technical, Harlow, 1990.

\end{thebibliography}
\def\cprime{$'$} \def\cprime{$'$} \def\cprime{$'$}

\end{document}